\documentclass[a4paper,12pt]{article}
\usepackage[utf8]{inputenc} 
\usepackage[T1]{fontenc}  
\usepackage{hyperref}     
\usepackage{url}           
\usepackage{booktabs}      
\usepackage{amsfonts}       
\usepackage{nicefrac}      
\usepackage{microtype}     
\usepackage{lipsum}
\usepackage{amssymb}
\usepackage{amsmath}
\usepackage{textcomp}
\usepackage{eurosym}
\usepackage{verbatim}
\usepackage{amsthm}
\usepackage{commath}
\usepackage{mathtools}
\usepackage{graphicx}
\usepackage{caption}
\usepackage{subcaption}
\usepackage{dsfont}
\usepackage{xcolor}
\usepackage[numbers]{natbib}
\usepackage{float}
\usepackage{wrapfig}
\usepackage{mathtools}
\usepackage{bigints}
\usepackage{breqn}
\usepackage{csquotes}
\MakeOuterQuote{"}

\theoremstyle{plain}
\newtheorem{thm}{Theorem}[section]

\newtheorem{coro}[thm]{Corollary}

\newtheorem{defn}[thm]{Definition}
\newtheorem{prop}[thm]{Proposition}
\theoremstyle{definition}
\newtheorem{ex}[thm]{Example}
\newtheorem{rem}[thm]{Remark}

\newcommand{\N}{\mathbb{N}}
\newcommand{\R}{\mathbb{R}}

\newcommand{\ind}{\perp\!\!\!\perp}
\def\e{\mathrm{e}}
\def\E{\mathbb{E}}
\def\d{\, \mathrm{d}}
\def\1{\mathds{1}}
\def\R{\mathbb{R}}
\def\P{\mathbb{P}}

\def\Q{\mathbb{Q}}
\def\E{\mathbb{E}}
\def\N{\mathbb{N}}

\def\X{\mathcal{X}}
\def\Y{\mathcal{Y}}
\def\K{\textup{\textbf{K}}}

\def\KL{\textup{ KL}}

\def\|{\, | \,}
\def\ee{\varepsilon}
\def\rd{\,\mathrm{d}}

\def\var{\mathrm{Var}}
\def\cov{\mathrm{Cov}}
\DeclareMathOperator*{\argmin}{arg\,min}
\DeclareMathOperator*{\argmax}{arg\,max}

\def\e{\mathrm{e}}

\def\F{\mathcal{F}}

\def\nn{\nonumber}

\def\eql{\overset{\mathcal L}{=}}

\numberwithin{equation}{section}
\begin{document}
\title{Information-Based Martingale Optimal Transport}
\author{Georges Kassis and Andrea Macrina\footnote{Corresponding author: a.macrina@ucl.ac.uk}  \\ \\
Department of Mathematics, University College London \\ London WC1E 6BT, United Kingdom}
\date{18 January 2026}
\maketitle
\vspace{-.5cm}
\begin{abstract}
Randomised arcade processes are a class of continuous stochastic processes that interpolate in a strong sense, i.e., omega by omega, between any given ordered set of random variables, at fixed pre-specified times. Utilising these processes as generators of partial information, a class of continuous-time martingale---the filtered arcade martingales (FAMs)---are constructed. FAMs interpolate through a sequence of target random variables, which form a discrete-time martingale. The research presented in this paper relaxes the FAM setting to the interpolation between probability measures instead and treats the problem of selecting the worst martingale coupling for given, convexly ordered, probability measures contingent on the paths of FAMs that are constructed using the martingale coupling. This optimisation problem, that we term the information-based martingale optimal transport problem (IB-MOT), can be viewed from different perspectives. It can be understood as a model-free construction of FAMs, in the case where the coupling is not determined {\it a priori}. It can also be considered from the vantage point of optimal transport (OT), where the problem is concerned with introducing a noise factor in martingale optimal transport, similarly to how the entropic regularisation of optimal transport introduces noise in OT. The IB-MOT problem is static in its nature, since its aim is to find a coupling. However, a corresponding dynamical solution can be found by considering the FAM constructed with the identified optimal coupling. The existence and uniqueness of its solution are shown and an algorithm for empirical measures is proposed.
\vspace{.25cm}
\end{abstract}
\\
\\
{\bf AMS2020}: 49Q22, 60G35, 60G44, 60J60, 62M20, 94A15.

\section{Introduction}
Filtered arcade martingales (FAMs), a new class of stochastic processes developed in \cite{KassisMacrina} (see also \cite{Kassis} for more details), are continuous-time martingales that interpolate almost surely between finitely many convexly ordered random variables---which form a discrete-time martingale referred to as the target martingale vector---by filtering information generated by a randomised arcade process (RAP). In this paper, we relax this setting into a weak interpolation problem. Instead of the vector $X$, we consider a vector of convexly ordered probability measures and aim to find a FAM that interpolates in law between these measures. Since a FAM interpolates almost surely, it also interpolates in law. But depending on the target measures, there may be several FAMs with different laws, while still sharing the same driver and interpolating coefficients in their RAPs, that solve this weak interpolation problem. Note that having this kind of choice in the almost sure interpolation setting is impossible, since once a target martingale vector $X$ is fixed and an $X$-RAP is chosen, an almost surely unique FAM fills the gaps between the components of the vector $X$.

This paper is about an optimisation method for selecting a "meaningful" FAM in the weak interpolation setting, for a fixed driver and interpolating coefficients. Selecting one of the FAMs is equivalent to selecting a martingale coupling for the target measures, as the coupling is the sole parameter distinguishing between the candidate FAMs. So, which FAM/coupling should be selected? Recall that a FAM estimates the value of the final target random variable by filtering the information that is unveiled over time by its RAP. In the absence of a given martingale coupling, one might be interested in the FAM that produces the "worst case scenario", i.e., the poorest possible estimation of the target measure over time, in some norm. In the spirit of this remark, we introduce in Section \ref{IBMOTformulation} the {\it information-based martingale optimal transport} (IB-MOT) problem, which selects the martingale coupling that maximises the expected cumulative weighted error between a target measure and its associated FAM. While this problem is here framed from a filtering perspective, it can be interpreted through various lenses. For instance, as the name suggests, the idea of selecting a coupling based on paths that connect its marginal distributions within the space of probability measures is reminiscent of optimal transport (OT). As we will explore in the next section, IB-MOT could be viewed as introducing noise in martingale optimal transport (MOT, \cite{Beiglbock}, \cite{Beiglbock2}), akin to how the entropic regularisation of OT introduces noise in OT. However, the nature of this noise injection differs here as it pertains to selecting a volatility rather than a drift. IB-MOT is fundamentally a static problem, that is, aiming to select an optimal coupling. However, since each coupling corresponds to a candidate FAM in law, IB-MOT also has a dynamical aspect. The solution to this dynamical formulation can be derived by inserting the static solution into the general integral representation of a FAM.

Before diving into IB-MOT, we start by recalling the central definitions and results from the arcade processes theory introduced in \cite{KassisMacrina}, and \cite{Kassis} for more details. In what follows, let $n\in \N_0$ and consider the collection of fixed dates $\{T_i \in \R \| i=0, 1, \ldots, n\}$ such that $0\leqslant T_0 <T_{1}<T_{2}<\ldots < T_n < \infty$, and introduce the ordered sets 
\begin{eqnarray*}
\{T_0,T_n\}_*&=&\{T_0, T_1, \ldots,T_n\},\\
(T_0,T_n)_*&=&\bigcup\limits_{i=0}^{n-1}(T_i,T_{i+1}).
\end{eqnarray*}
Let $(\Omega, \F, \mathbb{P})$ be a probability space, $\left(D_{t}\right)_{t \in [T_0,T_n]}$ a sample-continuous stochastic process such that $\P[D_t=0]<1$ whenever $t\in (T_0,T_n)_*$, and $X$ an $\R^{n+1}$-valued random vector that is independent of $(D_t)$.
\begin{defn} \label{interpolating}
The functions $f_0,f_1,\ldots, f_n$ are interpolating coefficients on $\{T_0,T_n\}_*$ if $f_0, $ $f_1,\ldots, f_n \in C^0\left([T_0,T_n],\R\right), f_i(T_i)=1,$ and $f_i(T_j)=0 \text{ for }i,j=0,\ldots,n,\, i \neq j$.
\end{defn}
We can now give the definition of an arcade process.
\begin{defn} \label{AP}
An arcade process (AP), denoted $(A_{t}^{(n)})_{t \in [T_0,T_n]}$, on the partition $\{T_0,T_n\}_*$ is a stochastic process of the form
\begin{equation}
A_{t}^{(n)}:=D_{t}-\sum_{i=0}^n f_i(t)D_{T_i},
\end{equation}
where $f_0,\ldots, f_n$ are  interpolating coefficients on $\{T_0,T_n\}_*$. The process $(D_t)_{t \in [T_0,T_n]}$ is the driver of the AP. We denote by $(\F^A_t)_{t \in [T_0,T_n]}$ the filtration generated by $(A_{t}^{(n)})_{t \in [T_0,T_n]}$.
\end{defn}
We observe that $A_{T_0}^{(n)}=A_{T_1}^{(n)}=\ldots=A_{T_n}^{(n)}=0$ by construction, for all $\omega \in \Omega$. That is, $(A_{t}^{(n)})_{t \in [T_0,T_n]}$ strongly interpolates between zeros at times $t \in \{T_0,T_n\}_*$. This idea is extended to the interpolation between the components of the random vector $X$ instead of interpolating between zeros.
\begin{defn} \label{RAP}
An $X$-randomised arcade process (X-RAP) $(I_t^{(n)})_{t \in [T_0,T_n]}$ on the partition $\{T_0,T_n\}_*$ is a stochastic process given by
\begin{equation}
    I_t^{(n)}:=S_t^{(n)}+  A_t^{(n)} = D_{t}-\sum_{i=0}^n \left(f_i(t)D_{T_i}-g_i(t)X_{i}\right),
\end{equation}
where $f_0,\ldots, f_n$ and $g_0,\ldots, g_n$ are interpolating coefficients on $\{T_0,T_n\}_*$. We refer to 
\begin{equation}
    S_t^{(n)}=\sum\limits_{i=0}^n g_i(t)X_{i}
\end{equation}
as the signal function of $I_t^{(n)}$ and to 
\begin{equation}
    A_t^{(n)}=D_{t}-\sum\limits_{i=0}^n f_i(t)D_{T_i}
\end{equation}
as the noise process of $I_t^{(n)}$. We denote by $(\F^I_t)_{t \in [T_0,T_n]}$ the filtration generated by $(I_t^{(n)})$.
\end{defn}
We emphasise that $I_{T_0}^{(n)}= X_{0}, \ldots, I_{T_n}^{(n)}= X_{n}$, so $(I_t^{(n)})$ is a strong stochastic interpolator between the random variables $X_{0},\ldots, X_{n}$. The signal function $(S_t^{(n)})$ is independent of the noise process $(A_t^{(n)})$ since $X$ is assumed independent of $(D_t)_{t\in [T_0,T_n]}$. We introduce the following notation:
\begin{enumerate}
    \item $\mathcal P(\R^n)$ is the set of Borel probability measures on $\R^n$, for $n\in \N_0$.
    \item $\mathcal P_1(\R^n)$ is the set of Borel probability measures on $\R^n$, for $n\in \N_0$, with finite first moment.
    \item $\Pi(\mu_0,\mu_1,\ldots, \mu_n):=  \{ \pi \in \mathcal P(\R^{n+1}) \| \mu_{i-1} \text{ is the ith marginal measure of } \pi, \text{ for all }  \\ i=1,\ldots, n+1 \}$ is the set of couplings of $\{\mu_0,\mu_1,\ldots, \mu_n\} \subseteq \mathcal P (\R)$.
    \item $\mathcal M (\mu_0,\mu_1,\ldots, \mu_n):=\{\pi \in \Pi(\mu_0,\mu_1,\ldots, \mu_n) \cap \mathcal P_1(\R^{n+1}) \| (X_0,X_1,\ldots, X_n) \sim \pi \implies \\ \E[X_n \| X_0,\ldots, X_m] = X_m \text{ for all } m=0,\ldots, n\}$ is the set of martingale couplings of $\{\mu_0,\mu_1,\ldots, \mu_n\} \subseteq \mathcal P (\R)$.
\end{enumerate}
We recall that $\mathcal M (\mu_0,\mu_1,\ldots, \mu_n) \neq \emptyset$ if and only if $\int_\R f(x) \d \mu_0(x) \leqslant \int_\R f(x) \d \mu_1(x)\leqslant \ldots \leqslant \int_\R f(x) \d \mu_n(x)$ for any convex function $f$ on $\R$, i.e, the measures $\mu_0,\mu_1,\ldots, \mu_n$ are convexly ordered, see \cite{Strassen} and \cite{Kellerer}. In this case, we write $\mu_0\leqslant_{\mathrm{cx}}\mu_1\leqslant_{\mathrm{cx}}\ldots\leqslant_{\mathrm{cx}} \mu_n$. We can now recall the definition of filtered arcade martingales (FAMs). We do this for $n=1$ because this is the case that is of interest in this paper. For $n>1$, we refer the reader to \cite{KassisMacrina}.

Let $X=(X_0,X_1)$ be a random vector distributed according to a martingale coupling $\pi^X \in \mathcal M (\mu_0,\mu_1)$, where $(\mu_0,\mu_1) \in \mathcal P_1 (\R) \times \mathcal P_1 (\R)$ and $\mu_0 \leqslant_{\mathrm{cx}}\mu_1$. This means that $X_0 \sim \mu_0$, $X_1 \sim \mu_1$, and $\E[X_1 \| X_0 ] = X_0$.
\begin{defn} \label{1FAM}
Given an $X$-RAP $(I_t^{(1)})$, a one-arc FAM for $X\sim \pi^X$ on $[T_0,T_1]$ is a stochastic process defined by $M_t=\E[X_1 \| \F_t^I ]$.
\end{defn}
The process $(M_t)$ is an $(\F_t^I)$-martingale that interpolates between $X_0$ and $X_1$, almost surely. The choice of the coupling $\pi^X$, alongside the interpolating coefficients and the driver of $(I_t^{(1)})$, determine the dynamics of the paths of a FAM, since they specify how $X_0$, $X_1$, and the information process are related to each other. When a FAM can be written as $M_t = \E [ X_1 \| X_0, I_t^{(1)} ] $, interesting properties arise. 

\begin{prop} \label{Ito}
Let $M_t = \E [ X_1 \| X_0, I_t^{(1)} ]$ be a one-arc FAM. If $(I_t^{(1)})$ is a semimartingale with Gauss-Markov driver $(D_t)$, such that $(t,x) \rightarrow V(t,x,X_0)/K(t,x,X_0)$ is $C^2(((T_0,T_1) \setminus N) \times Im(I^{(1)}))$ where $N \subset (T_0,T_1)$ contains finitely many elements, then there exists a standard Brownian motion on $[T_0,T_1]$ $(W_t)_{t\in[T_0,T_1]}$, called the innovations process, such that
\begin{equation} \label{GMFAM}
    M_t = X_0 + \int_{T_0}^t \frac{\var [ X_1 \| X_0, I_s^{(1)} ] \sqrt{H_1'(s)H_2(s)  - H_1(s) H_2'(s) }}{H_1(T_1)H_2(s)-H_1(s)H_2(T_1)} \d W_s
\end{equation}
for $t \in (T_0,T_1)$, where $H_1$ and $H_2$ are differentiable almost everywhere functions given by the covariance function of the driver: $K_D(x,y)=H_1(\min(x,y))H_2(\max(x,y))$.
\end{prop}
The notation $H'$ is short-hand for differentiation when the derivative exists almost everywhere. The innovations process $(W_t)$ can be obtained explicitly, see \cite{KassisMacrina}. In the next section, we introduce the information-based martingale optimal transport problem. 

\section{Simple analogy with several optimal transports} \label{analogy}
Let $(B_t)_{t\geqslant0}$ be a standard Brownian motion. For now, to compare with other interpolating methods, we start with the standard one-arc $X$-RAP $(I_t^{(1)})_{t\in[T_0,T_1]}$ given by 
\begin{equation} \label{RBB}
    I_{t}^{(1)}=  B_t - \frac{T_1-t}{T_1-T_0} (B_{T_0}-X_0) - \frac{t-T_0}{T_1-T_0} (B_{T_1}-X_1).
\end{equation}
We remark that the $n$-arc case, with a general standard RAP, can be treated similarly. We consider the FAM
\begin{equation}
    M_t= \E[ X_1 \mid \F_t^I]= \E[ X_1 \mid X_0, I_t^{(1)}] = X_0+ \int_{T_0}^t \frac{\var [ X_1 \mid X_0, I_s^{(1)} ]}{T_1-s} \rd W_s,
\end{equation}
where $(W_t)$ is the innovations process of $(M_t)$, a standard Brownian motion $[T_0,T_1]$ adapted to the filtration $(\F_t^I)$ that is generated by $(I_{t}^{(1)})$, which is $\{T_0,T_1\}_*$ -conditionally-Markov. The FAM $(M_t)$ depends on the coupling $\pi^X$ of $(X_0,X_1)$. In some real-world situations, the coupling is not observed directly, only its marginals are. Choosing a coupling is then part of the modelling task determined by the problem at hand. Another approach is to go ``model-free'' by utilising a least action principle, such as martingale optimal transport (MOT) for instance.

Optimal transport (OT) dates back to Gaspard Monge in 1781 \cite{Monge}, with significant advancements by Leonid Kantorovich in 1942 \cite{Kantorovich} and Yann Brenier in 1987 \cite{Brenier}. It provides a way to compare two measures, $\mu$ and $\nu$, defined on the Borel sets of topological Hausdorff spaces $\X$ and $\Y$, respectively. To make this statement precise, one needs a cost function $c: \X \times \Y \rightarrow [0,\infty]$ that measures the cost of transporting a unit mass from $x \in \X$ to $y\in\Y$. The optimal transport problem is concerned with how to transport $\mu$ to $\nu$ whilst minimising the cost of transportation. That is, given $\mu \in \mathcal{P}(\X)$ and $\nu \in \mathcal{P}(\Y)$,
\begin{equation}
    \inf_{\pi \in \Pi(\mu,\nu)} \K(\pi) := \inf_{\pi \in \Pi(\mu,\nu)} \int_{\X \times \Y} c(x,y) \rd \pi(x,y).
\end{equation}
This problem possesses many interesting properties (see for instance \cite{Villani1}, \cite{Villani2}, \cite{Bogachev}, or Section 2.3 in \cite{Kassis} for a short overview). For example, if $\X=\Y$ is a Polish space, and $(\X,c)$ is a metric space, then
\begin{equation}
W_{p}(\mu, \nu):=\left(\inf _{\pi \in \Pi(\mu, \nu)} \int_{\X \times \X} c(x, y)^{p} \mathrm{d} \pi(x, y)\right)^{1 / p}
\end{equation}
defines a metric for any $p\geqslant 1$, the Wasserstein $p$-metric, on the space $\mathcal P_p(\X)$ of probability measures on $\X$ with finite $p$th moment. Furthermore, if $\X$ is Euclidean, and $c(x,y)=\norm{x-y}$, there is a one-to-one correspondence between the minimisers $\pi^*$ of $\inf_{\pi \in \Pi(\mu,\nu)} \int_{\X \times \Y} c(x,y)^p \rd \pi(x,y)$, and the geodesics in $(\mathcal P_p(\X),W_{p})$: If $(X_0,X_1) \sim \pi^*$, the law of the process 
\begin{equation}
    \left(\frac{T_1-t}{T_1-T_0} X_0 + \frac{t-T_0}{T_1-T_0} X_1\right)_{t\in[T_0,T_1]}
\end{equation}
is the shortest paths from $\mu$ to $\nu$ in $(\mathcal P_p(\X),W_{p})$. This links optimal transport to interpolation on the space of random variables. The interpolating process is simple, since it is deterministic conditional on $(X_0,X_1)$. This is because OT does not incorporate noise in its formulation. If one wishes to introduce noise in the system, a solution could be to add an independent Brownian bridge on $[T_0,T_1]$ to $(T_1-t)/(T_1-T_0) X_0 + (t-T_0)/(T_1-T_0) X_1$. This is trivial on some level, since the noise term, here the Brownian bridge, is not related to the initial optimisation problem. We call such interpolator artificially noisy. If one wishes to interpolate while a noise term is involved within the optimal transport framework, the entropic regularisation of optimal transport \cite{Peyre1}, which we recall briefly, provides a way forward. Given $\mu \in \mathcal{P}(\X)$, $\nu \in \mathcal{P}(\Y)$, and $\ee >0$, one considers 
\begin{equation}
    \inf_{\pi_\ee \in \Pi(\mu,\nu)} \K_\ee(\pi_\ee) := \inf_{\pi_\ee \in \Pi(\mu,\nu)} \int_{\X \times \Y} c(x,y) \rd \pi_\ee(x,y) + \ee \KL(\pi_\ee \mid \mu \otimes \nu )
\end{equation}
where$\KL$ is the Kullback-Leibler divergence
\begin{equation}
    \KL(\pi_\ee \mid \mu \otimes \nu ) = \left\{
    \begin{array}{ll}
        \displaystyle{\int_{\X \times \Y} \ln \left(\frac{\rd\pi_\ee}{\rd \mu \rd \nu} (x,y)\right)} \rd\pi_\ee (x,y), & \mbox{if } \pi_\ee<< \mu \otimes \nu, \\
        \\
        \infty & \mbox{otherwise.} \\
    \end{array}
\right. 
\end{equation}
In the case where $\X=\Y$ is Euclidean, $c(x,y)=\norm{x-y}^2$, and $\ee = 2 (T_1-T_0)$, this problem is equivalent to Schrödinger's problem (\cite{Schrodinger}, \cite{Leonard1}) with volatility parameter equal to $1$, see \cite{Kassis} for an in-depth analysis. If $\pi_S^*$ is the solution to Schrödinger's problem, any stochastic process $(Y_t)_{t \in [T_0,T_1]}$ satisfying
\begin{equation}
    \P \, [\,  {\bf Y} \in \cdot \, ] = \int_{\X}\int_{\X} \P \, [ \, {\bf B} \in \cdot \mid B_{T_0} = x, \, B_{T_1} = y ] \rd\pi_S^* (x,y),
\end{equation}
where $(B_t)_{t \in \R}$ is a standard Brownian motion and
\begin{equation}
\begin{aligned}
    {\bf Y}: \quad &\Omega \rightarrow (C^0 ([T_0,T_1], \X),||.||_\infty) \\
    &\omega \mapsto \{Y_t(\omega) \mid t\in [T_0,T_1]\}
\end{aligned}
\end{equation}
\begin{equation}
\begin{aligned}
    \, \, {\bf B}: \quad &\Omega \rightarrow (C^0 ([T_0,T_1], \X),||.||_\infty), \\
    &\omega \mapsto \{B_t(\omega) \mid t\in [T_0,T_1]\}
\end{aligned}
\end{equation}
is called Schrödinger's bridge. Hence, $(I_{t}^{(1)})$ is an anticipative representation of Schrödinger's bridge with volatility parameter equal to $1$ (in dimension one, since we did not introduce a definition of RAPs in higher dimensions) as long as $(X_0,X_1) \sim \pi_S^*$. We call this interpolator genuinely noisy, as opposed to artificially noisy, because, in this case, the noise is taken into account in the optimisation problem. We summarise the three interpolators built using optimal transport in dimension one:
\begin{enumerate}
    \item The default OT interpolator, or the shortest path interpolator on $[T_0,T_1]$: $(T_1-t)/(T_1-T_0) X_0 + (t-T_0)/(T_1-T_0) X_1$ where $(X_0,X_1) \sim \pi^*$.
    \item The artificially noisy OT interpolator on $[T_0,T_1]$: $(I_{t}^{(1)})$, where $(X_0,X_1) \sim \pi^*$.
    \item The genuinely noisy OT interpolator, or Schrödinger's bridge on $[T_0,T_1]$: $(I_{t}^{(1)})$, where $(X_0,X_1) \sim \pi_S^*$.
\end{enumerate}
To redesign these interpolators to suit a martingale setup, let us fix $\X=\Y=\R$, which is the FAM setting. We can adapt optimal transport to yield a martingale coupling instead as discussed in Section 2.4: given $\mu,\nu \in \mathcal{P}_1(\R)$ in convex order,
\begin{equation}
    \inf_{\pi \in \mathcal M (\mu,\nu)} \K(\pi) = \inf_{\pi \in \mathcal M (\mu,\nu)} \int_{\R^2} c(x,y) \rd \pi(x,y),
\end{equation}
We denote the minimisers of this problem by $\pi_m^*$. There are many differences between optimal transport and its martingale counterpart. For instance, the most popular cost function $c(x,y)=(x-y)^2$ for optimal transport cannot be used in the martingale context because, in this case, the objective function $\K(\pi)$ does not depend on $\pi$:
\begin{align}
    \int_{\R^2} (x-y)^2 \rd \pi(x,y) &=  \int_{\R^2} x^2 + y^2 - 2xy \rd \pi(x,y) \nn \\
    &= \int_{\R^2} x^2 \rd \mu(x) + \int_{\R^2} y^2  \rd \nu(y) - 2 \int_{\R^2} y^2  \rd \nu(y) \nn \\
    &= \int_{\R^2} x^2 \rd \mu(x) - \int_{\R^2} y^2  \rd \nu(y).
\end{align}
Another difference is that we lose the geodesic interpretation, since there is no counterpart to the Wasserstein distance in the martingale context. So, what would be the martingale counterparts to the shortest path, artificially noisy, and genuinely noisy OT interpolators? This is not trivial, since the interpolators must be martingales, and $((T_1-t)/(T_1-T_0) X_0 +(t-T_0)/(T_1-T_0) X_1)_{t\in [T_0,T_1]}$ is not a martingale regardless of the coupling of $(X_0,X_1)$.

Since there is no longer something like a shortest path interpolator, we expect its counterpart in the martingale setup to be discontinuous. We propose the process that is equal to $X_0$ in $T_0$ and equal to $X_1$ for $t \in (T_0,T_1]$. For the artificially noisy MOT interpolator, the FAM $(M_t)$ with $(X_0,X_1) \sim \pi_m^*$ is a candidate. Indeed, it is a martingale, and its noise process is not taken into account in the selection of the optimal coupling $\pi_m^*$. Furthermore, if we remove the artificial noise, i.e., we set $A_t^{(1)}=0$ in the expression of $(M_t)$, and denote by $(\mathcal G_t)_{t\in[T_0,T_1]}$ the filtration generated by $((T_1-t)/(T_1-T_0) X_0 +(t-T_0)/(T_1-T_0) X_1)_{t\in[T_0,T_1]}$, we obtain 
\begin{equation}
    M_t= \E [ X_1 \mid \mathcal G_t] = \E \left[ X_1\, \bigg|\, X_0, \frac{T_1-t}{T_1-T_0} X_0 + \frac{t-T_0}{T_1-T_0} X_1 \right] = \left\{
    \begin{array}{ll}
          X_0   & \mbox{if } t=T_0,  \\
         X_1  & \mbox{if } t \in (T_0,T_1].
    \end{array}
\right.
\end{equation}
Thus, we retrieve the discontinuous MOT interpolator, similarly to how one retrieves the shortest path OT interpolator when setting the artificial noise to $0$ in the artificially noisy OT interpolator. For a genuinely noisy MOT interpolator, i.e., a sort of martingale counterpart to Schrödinger's bridge, we propose a new problem that we call the \emph{information-based martingale optimal transport problem}.

\section{Formulation of the problem and properties} \label{IBMOTformulation}
\begin{defn} \label{IBMOTdef}
Let $\mu$ and $\nu$ be two $L^2(\R)$-probability measures, in convex order. Given a martingale coupling $\pi \in \mathcal M (\mu,\nu)$, let $M_t(X_0,I_t^{(1)})=\E_\pi[X_1 \mid X_0, I_t^{(1)} ]$, $t\in [T_0,T_1]$ be the one-arc FAM for $(X_0,X_1)\sim \pi$ constructed using the randomised Brownian bridge $(I_t^{(1)})_{t\in [T_0,T_1]}$, see Eq. (\ref{RBB}). The information-based martingale optimal transport (IB-MOT) problem associated with $(I_t^{(1)})$ for the marginals $\mu$ and $\nu$ is 
\begin{equation}
    \sup_{\pi \in \mathcal M (\mu,\nu)} \E \left[ \int_{T_0}^{T_1} \frac{[X_1-M_t(X_0,I_t^{(1)})]^2}{T_1-t} \rd t \right].
\end{equation}
Similarly, for a general standard RAP $(I_t^{(1)})$ that satisfies the conditions in Proposition \ref{Ito}, the IB-MOT problem associated with $(I_t^{(1)})$ is
\begin{equation} \label{IB-MOT}
\begin{aligned}
& \sup_{\pi \in \mathcal M (\mu,\nu)} \K_I(\pi) := \\
& \sup_{\pi \in \mathcal M (\mu,\nu)} \E \left[ \int_{T_0}^{T_1} \frac{[X_1-M_t(X_0,I_t^{(1)})]^2 \sqrt{H_1'(t)H_2(t)  - H_1(t) H_2'(t) }}{H_1(T_1)H_2(t)-H_1(t)H_2(T_1)} \rd t \right]
\end{aligned}
\end{equation}
where the functions $H_1$ and $H_2$ are given by the covariance of the driver $(D_t)$ of $(I_t^{(1)})$. That is,  $K_D(x,y)= H_1 (\min(x,y)) H_2(\max(x,y))$ for all $(x,y) \in [T_0,T_1]^2$.
\end{defn}
As we shall see later, the supremum is always attained, and the solution $\pi_I^*$ is unique. The FAM associated with $\pi_I^*$, i.e., $M_t^*:=\E_{\pi^*_I}[X_1 \mid X_0, I_t^{(1)}]$, where $(X_0,X_1) \sim \pi_I^*$, is also referred to as the solution to the IB-MOT problem. The objective function $\K_I$ can be understood as the expected cumulative squared difference weighted error when estimating $X_1$ by the FAM $(M_t)$, where the weight function informs on the convergence rate of the FAM to its target. This weight is important for the selection of the optimal coupling since, by construction, $(M_t)$ converges to $X_1$. Hence, the errors towards the end of the interval are small, regardless of the coupling. To fairly compare the errors induced by a coupling at the beginning of the interval and those towards the end of the interval when averaging, the weight function is introduced, increasing the weight of the errors the closer they happen to time $T_1$.
\begin{rem}
IB-MOT is not simply MOT with a specific cost function. If one were to write $\K_I(\pi)$ as an integral of a function $c(x_0,x_1)$ against $\rd \pi (x_0,x_1)$,  the function would be 
\begin{dmath}
    c(x_0,x_1)=\E \left[ \int_{T_0}^{T_1} \frac{(x_1-M_t(x_0,g_0(t) x_0 + g_1(t) x_1 + A_t^{(1)}))^2 \sqrt{H_1'(t)H_2(t)  - H_1(t) H_2'(t) }}{H_1(T_1)H_2(t)-H_1(t)H_2(T_1)} \rd t \right],
\end{dmath}
where $g_0$ and $g_1$ are given by the signal function of $(I_t^{(1)})$. While $c(x_0,x_1)$ does indeed depend only on $x_0$ and $x_1$ explicitly, the expression of $(M_t)$ itself depends on $\pi$. Hence, $c(x_0,x_1)$ is not a suitable cost function for martingale optimal transport.
\end{rem}
We give another formulation of IB-MOT, by which we show that $\sup\limits_{\pi \in \mathcal M (\mu,\nu)} \K_I(\pi) < \infty$. 
\begin{prop} \label{IBMOTprop1}
It holds that
\begin{equation}
\begin{aligned}
\sup\limits_{\pi \in \mathcal M (\mu,\nu)} \K_I(\pi) & =\sup\limits_{\pi \in \mathcal M (\mu,\nu)} \E \left[ \int_{T_0}^{T_1} \frac{\var[X_1\mid X_0, I_t^{(1)}]\sqrt{H_1'(t)H_2(t)  - H_1(t) H_2'(t) }}{H_1(T_1)H_2(t)-H_1(t)H_2(T_1)} \rd t \right] \\
& <\infty .
\end{aligned}
\end{equation}
\end{prop}
\begin{proof}
Let $u_1(t)=\sqrt{H_1'(t)H_2(t)  - H_1(t) H_2'(t) }$ and $u_2(t)= H_1(T_1)H_2(t)-H_1(t)H_2(T_1)$. By the definition of the FAM and Itô's isometry, we have
\begin{align}
    \E[(X_1-X_0)^2] &= \E[(M_{T_1}-X_0)^2 ] \nn\\
    &= \E \left[ \left(\int_{T_0}^{T_1} \frac{\var[X_1\mid X_0, I_t^{(1)}]u_1(t)}{u_2(t)} \rd W_t\right)^2 \right] \nn\\
    &= \E \left[ \int_{T_0}^{T_1} \left(\frac{\var[X_1\mid X_0, I_t^{(1)}]u_1(t)}{u_2(t)}\right)^2 \rd t \right].
\end{align}
Since $\E[(X_1-X_0)^2 ]= \E[X_1^2] - 2 \E[X_0E[ X_1  \mid X_0] ] +\E[X_0^2 ] = \E[X_1^2] - \E[X_0^2 ] < \infty$, we have $\E [ \int_{T_0}^{T_1} (\var[X_1\mid X_0, I_t^{(1)}]u_1(t)/u_2(t))^2 \rd t ] <\infty$, and because the product space $\Omega \times [T_0,T_1]$ is of finite measure, by Hölder's inequality, we also have $\E [ \int_{T_0}^{T_1} |\var[X_1\mid X_0, I_t^{(1)}]u_1(t)/u_2(t)| \rd t ] <\infty$. Using the martingale property, we see that 
\begin{align}
    \E[\var[X_1\mid X_0, I_t^{(1)}]]&= \E [ \E [ X_1^2 \mid X_0, I_t^{(1)} ] - M_t^2 ] \nn\\
    &= \E [ X_1^2 ] - \E[M_t^2]\nn\\
    &= \E [ X_1^2+M_t^2] - 2 \E[M_t\,\E[X_1  \mid \F_t^I]]  \nn\\
    &= \E [ X_1^2+M_t^2] - 2 \E[\E[X_1 M_t \mid \F_t^I]]  \nn\\
    &= \E[X_1^2+M_t^2 - 2X_1 M_t] \nn\\
    &= \E [ (X_1-M_t)^2].
\end{align}
By Fubini's theorem, since $\E [ \int_{T_0}^{T_1} |\var[X_1\mid X_0, I_t^{(1)}]u_1(t)/u_2(t)| \rd t ] <\infty$, we get 
\begin{align}
    \E \left[ \int_{T_0}^{T_1} \frac{\var[X_1\mid X_0, I_t^{(1)}]u_1(t)}{u_2(t)} \rd t \right] &= \int_{T_0}^{T_1} \frac{\E \left[ \var[X_1\mid X_0, I_t^{(1)}]\right]u_1(t)}{u_2(t)} \rd t  \nn \\
    &= \int_{T_0}^{T_1} \frac{\E [ (X_1-M_t)^2]u_1(t)}{u_2(t)} \rd t.
\end{align}
Furthermore, also by Fubini's theorem, since 
\begin{equation}
    \int_{T_0}^{T_1}\E\left[  \abs{\frac{ (X_1-M_t)^2u_1(t)} {u_2(t)}}\right] \rd t =\int_{T_0}^{T_1}\E \left[  \abs{\frac{\var[X_1\mid X_0, I_t^{(1)}]u_1(t)}{u_2(t)}}\right] \rd t  <\infty,
\end{equation}
one has $\int_{T_0}^{T_1} \E [ (X_1-M_t)^2]u_1(t)/u_2(t) \rd t=\K_I(\pi)$. Hence, 
\begin{equation}
    \sup\limits_{\pi \in \mathcal M (\mu,\nu)} \K_I(\pi) = \sup\limits_{\pi \in \mathcal M (\mu,\nu)} \E \left[ \int_{T_0}^{T_1} \frac{\var[X_1\mid X_0, I_t^{(1)}]u_1(t)}{u_2(t)} \rd t \right] < \infty.
\end{equation}
\end{proof}
\begin{rem} \label{MBBrem}
Since the process $$\left(\frac{\var[X_1\mid X_0, I_t^{(1)}]\sqrt{H_1'(t)H_2(t)  - H_1(t) H_2'(t) }}{H_1(T_1)H_2(t)-H_1(t)H_2(T_1)}\right)_{t\in [T_0,T_1]}$$ is also the volatility process of $(M_t)$, see Eq. (\ref{GMFAM}), the IB-MOT problem can be seen as the martingale Benamou-Brenier problem (MBB, see \cite{Veraguas} and \cite{Kassis}) with the additional constraint that the volatility $(\sigma_t)_{t \in [T_0,T_1]}$ must be of the form
\begin{equation*}
\sigma_t=\frac{\var[X_1\mid X_0, I_t^{(1)}]\sqrt{H_1'(t)H_2(t)  - H_1(t) H_2'(t) }}{H_1(T_1)H_2(t)-H_1(t)H_2(T_1)}.
\end{equation*}
This also highlights the differences with Schrödinger's volatility models (SVMs) \cite{Labordere}, a martingale interpolating method that takes inspiration from the entropy minimising property of Schrödinger's problem and relies on a measure change: Consider two, possibly correlated, Brownian motions $(B_t)_{t \geqslant 0}$ and $(\tilde B_t)_{t \geqslant 0}$, under a measure $\P$, along with the system of SDEs on the real interval $[T_0,T_1]$
\begin{align}
&\d M_{t} = a_{t}\, M_{t}\d\widetilde B_t, \nn\\
&\d a_{t}=b(a_t)+ c\left(a_{t}\right) \text{d} B_t,
\end{align}
where the real functions $b$ and $c$ are given, and are such that $(a_t)_{t \in [T_0,T_1]}$ is an Itô process. The goal is to change the measure $\P$ to an equivalent one, such that the process $(M_t)_{0 \leqslant t \leqslant 1}$ matches the given measures $\mu$ and $\nu$ at times $T_0$ and $T_1$, respectively, while being a true martingale. If several equivalent measures exist that satisfy these constraints, one must select the one that is closest to $\P$ in terms of the Kullback-Leibler divergence. By Girsanov's theorem, changing $\P$ to an equivalent measure $\Q$ transforms the system of SDEs as follows: 
\begin{align}
&\d M_{t}= a_{t} M_{t} (\text{d}\widetilde B_t + \widetilde \lambda_t \d t), \nn\\
&\d a_{t}=b(a_t)+ c\left(a_{t}\right) (\text{d} B_t + \lambda_t \d t).
\end{align}
One must then find the drifts $(\tilde \lambda_t)$ and $(\lambda_t)$, generated by the same measure change, that minimise the functional $\E_\Q [ \int_{T_0}^{T_1} \widetilde{\lambda}^2_s + \lambda^2_s \d s]$ while satisfying the constraints $M_0 \sim \mu$ and $M_1 \sim \nu$, and so that $(M_t)_{t\in [T_0,T_1]}$ be a martingale.
\end{rem}

Eq. (\ref{IB-MOT}) can be written more concisely using the integration by parts formula:
\begin{coro} \label{IBMOTcoro1}
It holds that 
\begin{equation}
    \sup\limits_{\pi \in \mathcal M (\mu,\nu)} \K_I(\pi) = \sup\limits_{\pi \in \mathcal M (\mu,\nu)} \E[X_1 W_{T_1}].
\end{equation}
\end{coro}
\begin{proof}
We recall that
\begin{equation}
    M_t = X_0 + \int_{T_0}^t \frac{\var [ X_1 \mid X_0, I_s^{(1)} ] \sqrt{H_1'(s)H_2(s)  - H_1(s) H_2'(s) }}{H_1(T_1)H_2(s)-H_1(s)H_2(T_1)} \rd W_s,
\end{equation}
where $(W_t)$ is the innovations process of $(M_t)$, an $(\F^I_t)$-adapted standard Brownian motion on $[T_0,T_1]$. Integrating by parts, we obtain
\begin{dmath}
    \E \left[ \int_{T_0}^{T_1} \frac{\var [ X_1 \mid X_0, I_s^{(1)} ] \sqrt{H_1'(s)H_2(s)  - H_1(s) H_2'(s) }}{H_1(T_1)H_2(s)-H_1(s)H_2(T_1)} \rd t \right] = \E[M_{T_1}W_{T_1} - M_{T_0}W_{T_0}] = \E [X_1 W_{T_1}],
\end{dmath}
which confirms the statement of the corollary. 
\end{proof}
 Because $W_{T_1}$ is a $\mathcal N(0,T_1-T_0)$-distributed random variable, regardless of the coupling $\pi$, one can complete the square to obtain an equivalent problem to Eq. (\ref{IB-MOT}):
\begin{align}
    \E[X_1^2] - 2 \sup\limits_{\pi \in \mathcal M (\mu,\nu)} \E [X_1 W_{T_1}] + \E[W_{T_1}^2] &= \inf\limits_{\pi \in \mathcal M (\mu,\nu)} \E[X_1^2] - 2\E [X_1 W_{T_1}] + \E[W_{T_1}^2] \nn \\
    &= \inf\limits_{\pi \in \mathcal M (\mu,\nu)} \E[(X_1 - W_{T_1})^2].
\end{align}
This equivalent formulation of IB-MOT highlights the similarities and differences with MOT. We retrieve the ability to use the difference squared as a cost, and noise has been introduced inside the objective function by replacing what in MOT is usually $X_0$ with $W_{T_1}$. 
\begin{rem}
It is important to note that, even though the innovations process is a Brownian motion regardless of the coupling $\pi$ and the choice of the RAP $(I_t^{(1)})$ that satisfies the conditions in Proposition \ref{Ito}, the joint distribution of the innovations process and $X_1$ does depend on $\pi$ and $(I_t^{(1)})$. This can be observed by investigating the general almost sure expression of $(W_t)$, see \cite{KassisMacrina}, given by $W_t= \int_{T_0}^t 1/\sqrt{(H_1'(s)H_2(s) - H_1(s) H_2'(s))} \rd N_s$, where
\begin{dmath}
    \rd N_t{:=} \left( \frac{[ H_1'(t)H_2(T_1)  - H_1(T_1) H_2'(t)]Z_t - [H_1'(t)H_2(t)  - H_1(t) H_2'(t)]M_t}{H_1(T_1)H_2(t) - H_1(t) H_2(T_1) }  - J_t \right) \rd t + \rd I^{(1)}_t, \nn
\end{dmath}
$Z_t:= I_t^{(1)} -g_0(t)X_0  - \E[A^{(1)}_t]$ and $J_t:= X_0\, g_0'(t) + (\E[A^{(1)}_t])'$.
\end{rem}
\begin{rem}
Although we focus on a formulation of IB-MOT over a single period, all the statements and proofs up to this point still hold with minor changes, detailed below, for the following formulation of the multi-period IB-MOT problem. Let $\mu_0$, $\mu_1$, \ldots, $\mu_n$ be $L^2$-probability measures on $\R$ in convex order. Given a martingale coupling $\pi \in \mathcal M (\mu_0,\mu_1,\ldots, \mu_n)$, let $(I_t^{(n)})$ be a suitable n-arc $X$-RAP on $\{T_0,T_n\}_*$, see \cite{KassisMacrina} and \cite{Kassis} for further details, and $M_t= \E[ X_n \mid X_0,\ldots, X_{m(t)}, I_t^{(n)} ]$ an $n$-arc FAM for $X=(X_0,X_1, \ldots, X_n) \sim \pi$. The multi-period IB-MOT problem associated with $(I_t^{(n)})$ for the measures $(\mu_0$, $\mu_1$, \ldots, $\mu_n)$ is
\begin{equation}
\sup_{\pi \in \mathcal M (\mu_0,\mu_1,\ldots, \mu_n)} \K_I^{(n)}(\pi),
\end{equation}
where
\begin{equation}
\K_I^{(n)}(\pi)=\E \left[ \int_{T_0}^{T_n} \frac{[X_n-M_t(X_0,\ldots, X_{m(t)}, I_t^{(n)})]^2 \sqrt{H_1'(t)H_2(t)  - H_1(t) H_2'(t) }}{H_1(T_n)H_2(t)-H_1(t)H_2(T_n)} \rd t \right].
\end{equation}
It suffices to replace $X_1$ by $X_n$, $T_1$ by $T_n$, and condition on $(X_0,\ldots, X_{m(t)}, I_t^{(n)})$ instead of $(X_0, I_t^{(1)})$ in all statements and proofs to obtain their multi-period equivalent:
\begin{align}
&\sup_{\pi \in \mathcal M (\mu_0,\mu_1,\ldots, \mu_n)} \K_I^{(n)}(\pi) \nn \\ 
&= \sup_{\pi \in \mathcal M (\mu_0,\mu_1,\ldots, \mu_n)} \E \left[ \int_{T_0}^{T_n} \frac{\var[X_n\mid X_0, \ldots, X_{m(t)}, I_t^{(n)}]\sqrt{H_1'(t)H_2(t)  - H_1(t) H_2'(t) }}{H_1(T_n)H_2(t)-H_1(t)H_2(T_n)} \rd t \right] \nn \\ 
&= \sup_{\pi \in \mathcal M (\mu_0,\mu_1,\ldots, \mu_n)} \E[X_n W_{T_n}] < \infty.
\end{align}
\end{rem}
Next, we show that a solution to an IB-MOT problem always exists in $\mathcal M (\mu,\nu)$, and it is unique.
\begin{prop} \label{IB-MOTsol}
It holds that 
\begin{equation}
    \sup\limits_{\pi \in \mathcal M (\mu,\nu)} \K_I(\pi) = \max\limits_{\pi \in \mathcal M (\mu,\nu)} \K_I(\pi)
\end{equation} 
and the solution $\pi^*_I :=\argmax\limits_{\pi \in \mathcal M (\mu,\nu)} \K_I(\pi)$ is unique.
\end{prop}
\begin{proof}
Let $\mathcal M_x (\mu,\nu) = \{\gamma_x \in \mathcal P(\R) \mid \rd\pi(x,y) = \rd \gamma_x (y) \rd \mu (x), \pi \in \mathcal M (\mu,\nu)\}$ be the set of conditional martingale measures for $\mu$ and $\nu$. First, we show that 
\begin{equation}
    \inf\limits_{\pi \in \mathcal M (\mu,\nu)} \E\left[(X_1 - W_{T_1})^2\right] = \inf\limits_{\gamma_x \in \mathcal M_x (\mu,\nu)} \int_\R \int_0^1( Q_{\gamma_x}(\alpha)-Q_{\mathcal N_x}(\alpha))^2 \rd \alpha \rd \mu(x),
\end{equation}
where $Q_{\mathcal N_x}$ is the conditional quantile function of $W_{T_1}$ given $X_0=x$, and $Q_{\gamma_x}$ is the conditional quantile function of $X_1$ given $X_0=x$. We have:
\begin{enumerate}
\item For any $\gamma_x \in \mathcal M_x (\mu,\nu)$, there exists a $\pi \in \mathcal M(\mu,\nu)$, satisfying $\rd \pi(x,y) = \rd \gamma_x(y) \rd\mu(x)$, such that $\int_\R \int_0^1( Q_{\gamma_x}(\alpha)-Q_{\mathcal N_x}(\alpha))^2 \rd \alpha \rd \mu(x) = \E_\pi[(X_1 - W_{T_1})^2]$. So,
    \begin{equation}
        \inf\limits_{\pi \in \mathcal M (\mu,\nu)} \E\left[(X_1 - W_{T_1})^2\right] \leqslant \inf\limits_{\gamma_x \in \mathcal M_x (\mu,\nu)} \int_\R \int_0^1( Q_{\gamma_x}(\alpha)-Q_{\mathcal N_x}(\alpha))^2 \rd \alpha \rd \mu(x).
    \end{equation}
\item For any $\pi \in \mathcal M(\mu,\nu)$, there exists a $\tilde \pi_x \in \Pi(\gamma_x, \mathcal N_x)$ such that $\E_\pi[(X_1 - W_{T_1})^2] = \E [\E[(X_1 - W_{T_1})^2 \mid X_0] ]= \int_\R \int_{\R^2} (x_1 - w)^2 \rd \tilde \pi_x (x_1,w) \rd \mu(x)$. So,
\begin{align}
    &\inf\limits_{\pi \in \mathcal M (\mu,\nu)} \E\left[(X_1 - W_{T_1})^2\right] \nn \\
    &\geqslant \inf\limits_{\gamma_x \in \mathcal M_x (\mu,\nu)} \int_\R \inf\limits_{\pi_x \in \Pi (\gamma_x, \, \mathcal N_x)} \int_{\R^2} (x_1 - w)^2 \rd  \pi_x (x_1,w) \rd \mu(x) \nn \\
    &=\inf\limits_{\gamma_x \in \mathcal M_x (\mu,\nu)} \int_\R \int_0^1( Q_{\gamma_x}(\alpha)-Q_{\mathcal N_x}(\alpha))^2 \rd \alpha \rd \mu(x).
\end{align}
\end{enumerate}
Hence, we have equality, and, if one of the infimums is attained, both are, while satisfying $\rd \pi(x,y) = \rd \gamma_x(y) \rd \mu(x)$. Since $\mathcal M (\mu,\nu)$ is weakly compact in $\mathcal P(\R^2)$, any sequence in $\mathcal M (\mu,\nu)$ admits a weakly converging subsequence with limit in $\mathcal M (\mu,\nu)$. Let us consider a minimising sequence, and furthermore denote a weakly converging subsequence by $(\pi_n) \subseteq \mathcal M (\mu,\nu)$. Using the relationship $\rd \pi(x,y) = \rd \gamma_x(y) \rd \mu(x)$, we construct another weakly converging sequence $(\gamma_{x,n}) \subseteq \mathcal M_x (\mu,\nu)$ with limit inside $\mathcal M_x (\mu,\nu)$. We denote its limit by $\gamma_x^*$. We have
\begin{align}
    &\inf\limits_{\gamma_x \in \mathcal M_x (\mu,\nu)} \int_\R \int_0^1( Q_{\gamma_x}(\alpha)-Q_{\mathcal N_x}(\alpha))^2 \rd \alpha \rd \mu(x) \leqslant  \int_\R \int_0^1( Q_{\gamma_x^*}(\alpha)-Q_{\mathcal N_x}(\alpha))^2 \rd \alpha \rd \mu(x) \nn \\
    & \hspace{6cm} \leqslant \int_\R \liminf_{n \rightarrow \infty} \int_0^1( Q_{\gamma_{x,n}}(\alpha)-Q_{\mathcal N_x}(\alpha))^2 \rd \alpha \rd \mu(x),
\end{align}
where we used the lower semicontinuity of the Wasserstein metric on $\mathcal P (\R)$, see Lemma 4.3 and Remark 6.12 in \cite{Villani2}. By applying Fatou's lemma to $x \mapsto \int_0^1( Q_{\gamma_{x,n}}(\alpha)-Q_{\mathcal N_x}(\alpha))^2 \rd \alpha$, we have
\begin{align}
        &\int_\R \liminf_{n \rightarrow \infty} \int_0^1( Q_{\gamma_{x,n}}(\alpha)-Q_{\mathcal N_x}(\alpha))^2 \rd \alpha \rd \mu(x) \nn \\
        &\hspace{6cm} \leqslant \lim_{n \rightarrow \infty}\int_\R  \int_0^1( Q_{\gamma_{x,n}}(\alpha)-Q_{\mathcal N_x}(\alpha))^2 \rd \alpha \rd \mu(x)\nn \\
        &\hspace{6cm} = \inf\limits_{\gamma_x \in \mathcal M_x (\mu,\nu)} \int_\R \int_0^1( Q_{\gamma_x}(\alpha)-Q_{\mathcal N_x}(\alpha))^2 \rd \alpha \rd \mu(x).
\end{align}
Hence, 
\begin{equation}
    \lim_{n \rightarrow \infty}\int_\R  \int_0^1( Q_{\gamma_{x,n}}(\alpha)-Q_{\mathcal N_x}(\alpha))^2 \rd \alpha \rd \mu(x)=\int_\R \int_0^1( Q_{\gamma_x^*}(\alpha)-Q_{\mathcal N_x}(\alpha))^2 \rd \alpha \rd \mu(x),
\end{equation}
and the infimum is attained at $\gamma_x^* \in \mathcal M_x (\mu,\nu)$. Thus, the coupling $\pi^* \in \mathcal M (\mu ,\nu)$ defined by $\rd \pi^* (x,y) = \rd \gamma_x^*(y) \rd \mu (x)$ is equal to $\argmin_{\pi \in \mathcal M (\mu,\nu)} \E[(X_1 - W_{T_1})^2]$ and $\argmax_{\pi \in \mathcal M (\mu,\nu)} \K_I(\pi)$. Uniqueness follows from the strict convexity of the map $\gamma_x \mapsto \int_0^1( Q_{\gamma_x}(\alpha)-Q_{\mathcal N_x}(\alpha))^2 \rd \alpha$.
\end{proof}
\begin{ex} \label{IBMOTex1}
Let $T=T_1-T_0, \sigma>0, X_0 \sim \mathcal N(0,\sigma^2), X_1 \sim \mathcal N(0,\sigma^2 +T )$. We show that $$\pi=\mathcal N \left( \begin{pmatrix} 0 \\ 0  \end{pmatrix}, \begin{pmatrix} \sigma^2 & \sigma^2\\ \sigma^2 & \sigma^2+T \end{pmatrix}\right),$$ which we call {\it Brownian coupling}, is the solution to the IB-MOT problem in the case
\begin{equation}
    I_{t}^{(1)}= B_t - \frac{T_1-t}{T_1-T_0} (B_{T_0}-X_0) - \frac{t-T_0}{T_1-T_0} (B_{T_1}-X_1).
\end{equation}
We recall that $\sup\limits_{\pi \in \mathcal M (\mu,\nu)} \K_I(\pi) \leqslant \E[X_1^2]-\E[X_0^2] = T$. Assuming that $$ (X_0,X_1) \sim \mathcal N \left( \begin{pmatrix} 0 \\ 0  \end{pmatrix}, \begin{pmatrix} \sigma^2 & \sigma^2\\ \sigma^2 & \sigma^2+T \end{pmatrix}\right),$$
we obtain:
\begin{enumerate}
    \item The triplet $(X_0,X_1,I_s^{(1)})$ is Gaussian for all $s\in [T_0,T_1]$, and $(I_t^{(1)})$ is a standard Brownian motion on $[T_0,T_1]$.
    \item $\cov[X_0,I_s^{(1)}]= \sigma^2$.
    \item $\cov[X_1,I_s^{(1)}]= \frac{T_1-s}{T} \sigma^2 + \frac{s-T_0}{T} (\sigma^2+T) = \sigma^2 + s-T_0$.
    \item \begin{dmath*}
        \var[I_s^{(1)}]= \frac{(T_1-s)^2 }{T^2}\sigma^2 + \frac{(s-T_0)^2}{T^2} (\sigma^2+T) + 2 \frac{(T_1-s)(s-T_0)}{T} \sigma^2 + \frac{(T_1-s)(s-T_0)}{T} 
        = (s-T_0)^2 + \sigma^2T + \frac{(T_1-s)(s-T_0)}{T}.
        \end{dmath*}
    \item \begin{align}
        M_s &= \begin{pmatrix} \sigma^2  & \sigma^2 + s-T_0  \end{pmatrix} \begin{pmatrix} \sigma^2  & \sigma^2 \\ \sigma^2 & (s-T_0)^2 + \sigma^2T + \frac{(T_1-s)(s-T_0)}{T} \end{pmatrix}^{-1}\begin{pmatrix} X_0 \\ I_s^{(1)}  \end{pmatrix}, \nn \\
        &= I_s^{(1)} \nn.
        \end{align}
    \item \begin{align}
    \K_I \left(\mathcal N \left( \begin{pmatrix} 0 \\ 0  \end{pmatrix}, \begin{pmatrix} \sigma^2 & \sigma^2\\ \sigma^2 & \sigma^2+T \end{pmatrix}\right)\right) &= \E \left[ \int_{T_0}^{T_1} \frac{(X_1 - I_s^{(1)})^2}{T_1-s} \rd s \right], \nn \\
    &= \E \left[ \int_{T_0}^{T_1} \frac{\frac{(T_1-s)^2}{T^2}(X_1-X_0)^2 + \left(A_s^{(1)}\right)^2}{T_1-s} \rd s \right], \nn \\
    &= \int_{T_0}^{T_1} \frac{(T_1-s)^2 + (T_1-s)(s-T_0)}{T(T_1-s)} \rd s, \nn \\
    &=T\nn .
    \end{align}
\end{enumerate}
Hence, the Brownian coupling is the optimal martingale coupling for $(X_0,X_1)$, and Brownian motion is the optimal FAM between $X_0$ and $X_1$, according to IB-MOT.
\end{ex}
\begin{rem}
    As discussed in Remark \ref{MBBrem}, there seems to be a connection, or at least an analogy, between FAMs \& IB-MOT and standard stretched Brownian motions \& MBB. The goal of this remark is to discuss their relationship and formulate an open question. We set $T_0=0, T_1=1$ for simplicity. Let $\mu$ and $\nu$ be two $L^2(\R)$-probability measures, in convex order. We recall the definitions of a FAM and of a standard stretched Brownian motion. 
    \begin{itemize}
        \item Standard stretched Brownian motion (ssBm): Given a standard Brownian motion $(B_t)$, a process $(S_t)_{t\in [0,1]}$ is a $(\mu,\nu)$-ssBm on $[0,1]$ if there exists a random variable $Y\ind (B_t)$ and a function $F:\R \rightarrow \R$, defined as the gradient of a convex function and satisfying $F(B_1 + Y) \sim \nu$, such that 
    \begin{equation}
        \E[ F(B_1 + Y) \mid Y] \sim \mu, \quad S_t\eql \E[ F(B_1 + Y) \mid \F_t],
    \end{equation}
    for all $t\in[0,1]$, where $(\F_t)$ is the filtration generated by $(B_t+Y)$. The existence and uniqueness of the function $F$ and the distribution of $Y$ are guaranteed when $(\mu,\nu)$ is irreducible\footnote{We say that the pair $(\mu, \nu)$ is irreducible if for all measurable sets $A, B \subseteq \R$ with $\mu(A)>0$ and $\nu(B)>0$, there exists a martingale coupling $\pi \in \mathcal M(\mu, \nu)$ such that $\pi(A \times B)>0$.} in the sense of \cite{Beiglbock}, Appendix A.1.
    \item FAM: Given an $X$-RAP $(I_t^{(1)})$, a one-arc FAM for $X\sim \pi^X \in \mathcal M(\mu,\nu)$ on $[0,1]$ is a stochastic process defined by $M_t=\E[X_1 \mid \F_t^I ]$.
    \end{itemize}
    A FAM exists regardless of the irreducibility of $(\mu,\nu)$. For a fixed driver and interpolating coefficients, there are as many FAMs in law for $(\mu,\nu)$ as there are elements in $\mathcal{M}(\mu,\nu)$.

    To compare these two types of stochastic processes, i.e., ssBMs and FAMs, let us suppose that $(\mu,\nu)$ is irreducible so that the $(\mu,\nu)$-ssBm exists and is unique in law. Since an ssBm is obviously related to Brownian motion, let us fix the driver $(D_t)$ of the RAP to be a standard Brownian motion and the interpolating coefficients to be $f_0(t)=g_0(t)=1-t$, $f_1(t)=g_1(t)=t$. An open question is: among all the FAMs for $(\mu,\nu)$ in law, is there one that coincides in law with the $(\mu,\nu)$-ssBm? This is not a trivial question: what enables the interpolation by a FAM is the filtration $(\F_t^I)$, while in the ssBm case, it is the function $F$ and the distribution of $Y$. Put it another way, we ask whether there is a coupling $\pi^X \in \mathcal M (\mu,\nu)$ under which 
    \begin{equation}
        \E_{\pi^X}\left[X_1 \mid X_0, (1-t) X_0 + t (X_1-D_1) + D_t\right]\eql \E[ F(B_1 + Y) \mid Y, Y+B_t].
    \end{equation}
    We do not have the answer to this question. However, we know that for some choices of $(\mu,\nu)$, we can indeed have both processes match in law. For instance, when $\mu = \mathcal N (0, 1)$ and $\nu = \mathcal N (0, 2)$, the ssBm is a Brownian motion restricted to $[1,2]$ shifted to $[0,1]$, which is also a FAM with the coupling $$\pi^X=\mathcal N \left( \begin{pmatrix} 0 \\ 0  \end{pmatrix}, \begin{pmatrix} 1 & 1 \\ 1 & 2 \end{pmatrix}\right),$$ see Example \ref{IBMOTex1}. This is because ssBms and FAMs---driven by Brownian motion with the fixed interpolating coefficients above--- are both trying to mimic Brownian motion in a way, and in this particular case, both of them are able to match it. When a FAM has a different driver and interpolating coefficients than the ones discussed here, there is no reason to believe that there would be a coupling that makes it match the law of the ssBm, since the FAM is no longer trying to mimic Brownian motion. 

    Now, in the case of the standard Brownian driver with interpolating coefficients $f_0(t)=g_0(t)=1-t$, $f_1(t)=g_1(t)=t$, suppose that the answer to the open question is affirmative, or at least that there exists a subset of irreducible measures $(\mu,\nu)$ large enough (larger than the Gaussian measures) for which the answer is affirmative. Could it be that the special FAM that matches in law the ssBm (which is the solution to MBB in the case of irreducible measures) is the solution of IB-MOT? In other words, when the ssBm is a FAM in law, do IB-MOT and MBB share their solution in law? This is the case for Gaussian measures, which are the only measures for which we know the explicit solution to both IB-MOT and MBB. Furthermore, IB-MOT and MBB look "alike" in their two main forms, as compared in Table \ref{IBMOT-MBB}.    
    \begin{table}[H]
    \begin{center}
    \begin{tabular}{|p{9cm}|p{5cm}|}
    \hline
    IB-MOT &  MBB \\
    \hline
    $$\sup_{\pi} \E \left[\int_{0}^{1} \sigma_{t} \rd t\right]$$ where \newline $\sigma_t=\var[X_1\mid X_0, (1-t) X_0 + t (X_1-D_1) + D_t]/(1-t)$ & $$\sup_{\sigma_t} \E \left[\int_{0}^{1} \sigma_{t} \rd t\right]$$ such that \newline $Z_t= Z_0+ \int_0^t \sigma_t \rd B_t$, $Z_0 \sim \mu$, $Z_1 \sim \nu$  \\
    \hline
    $$\sup_\pi \E [X_1 W_{1}]$$ where \newline $W_1=\int_{0}^1 (X_1  - M_u  +D_u - uD_1 )/(1-u) \rd u  $, \newline $M_u= \E_\pi\left[X_1 \mid X_0, (1-u) X_0 + u (X_1-D_1) + D_u\right]$ & $$\sup_{\sigma_t} \E[Z_1 B_1]$$ such that \newline $Z_t= Z_0+ \int_0^t \sigma_t \rd B_t$, $Z_0 \sim \mu$, $Z_1 \sim \nu$ \\
    \hline
    \end{tabular}
    \caption{Comparison between characteristics of IB-MOT and MBB.}
    \label{IBMOT-MBB}
    \end{center}
    \end{table}
    As shown in Table \ref{IBMOT-MBB}, IB-MOT is a more constrained version of MBB. But in the case where the ssBm is a FAM in law, the constraint on the volatility in IB-MOT will automatically be satisfied in MBB, and so IB-MOT and MBB will share their solution in law. Hence, answering the open question would tell us exactly when IB-MOT and MBB share their solution in law. We summarise this discussion in a few items:
    \begin{description}
        \item [Item 1] When $(\mu, \nu)$ is irreducible, is the $(\mu, \nu)$-ssBm a FAM driven by standard Brownian motion with $f_0(t)=g_0(t)=1-t$, $f_1(t)=g_1(t)=t$ in law? If so, IB-MOT (for the fixed driver and interpolating coefficients) and MBB share their solution whenever $(\mu, \nu)$ is irreducible.
        \item [Item 2] If the answer to Item 1 is negative, is there a subset of the irreducible measures, larger than the Gaussian measures, for which the $(\mu, \nu)$-ssBm is a FAM driven by standard Brownian motion with $f_0(t)=g_0(t)=1-t$, $f_1(t)=g_1(t)=t$ in law? If so, IB-MOT (for the fixed driver and interpolating coefficients) and MBB share their solution for that subset of measures.
        \item [Item 3] If the answer to Item 2 is negative, then IB-MOT (for the fixed driver and interpolating coefficients) and MBB only share their solution in law on the set of Gaussian measures.
    \end{description}
    Note that when $(\mu,\nu)$ is not irreducible, IB-MOT and MBB cannot share their solution in law since the solution of MBB in that case would be a stretched Brownian motion (sBm), i.e., an ssBm on each irreducible component of $(\mu,\nu)$, whereas the solution to IB-MOT has the same expression on all irreducible components of $(\mu,\nu)$. For more on the structure of MBB and on the decomposition of sBm into ssBms, we refer to \cite{Veraguas2} and \cite{Schachermayer}.

    A potential avenue that could shed some light on the open question is the study of IB-MOT when the target measures are lognormal in convex order. This is the only non-Gaussian case where MBB has an explicit analytical solution: the (martingale) geometric Brownian motion, see \cite{Veraguas} Remark 1.9 and \cite{Veraguas3} Proposition 7.1. Geometric Brownian motion is a special process in "BB-theories". In \cite{Veraguas3}, the authors introduced an analogous problem to MBB: the geometric MBB (GMBB), where the optimisation takes place over the set of martingales of the form $Z_t= Z_0+ \int_0^t \sigma_t Z_t \rd B_t$, $Z_0 \sim \mu$, $Z_1 \sim \nu$, instead. The solution to GMBB is a martingale interpolating between the measures $\mu$ and $\nu$ that mimics geometric Brownian motion instead of Brownian motion. The authors show that geometric Brownian motion is the only common solution of MBB and GMBB. Going back to our open question, if we could show that geometric Brownian motion is a FAM for lognormal targets, then it would be the solution to IB-MOT, since it is solution to MBB. This would be an example of non-Gaussian targets where IB-MOT and MBB share their solution in law. 

    Let $S_t=\exp(B_t - t/2)$ and $(X_0,X_1)=(S_1,S_2)$. We recall that $S_t \sim$ lognormal$(-t/2, t)$, $\E[S_t]= 1$, $\var[S_t]= \e^t -1$, and $S_t \, | \, S_u \sim$ lognormal$(\ln S_u-(t-u)/2, t-u)$ for $0 < u < t$. This means that 
    \begin{equation}
        f^{X_1 \mid X_0}(x)=\frac{1}{x \sqrt{2 \pi}} \exp\left(-\frac{1}{2}\left(\ln \left(\frac{x}{X_0}\right)+\frac{1}{2}\right)^2\right).
    \end{equation}
    Then, the FAM for $(X_0,X_1)$ driven by standard Brownian motion $(D_t)$ with $f_0(t)=g_0(t)=1-t$, $f_1(t)=g_1(t)=t$ is
    \begin{align}
        M_t&=\E\left[X_1 \mid X_0, (1-t) X_0 + t (X_1-D_1) + D_t\right] \nn \\
        &= \frac{\int_0^\infty y f^{A^{(1)}_t} (A^{(1)}_t - g_1(t) (y-X_1))  f^{X_1 \mid X_0}(y) \rd y }{\int_0^\infty  f^{A^{(1)}_t} (A^{(1)}_t - g_1(t) (y-X_1))  f^{X_1 \mid X_0}(y) \rd y }\nn \\
        &= \frac{\int_0^\infty  \exp[{-\frac{(D_t - tD_1 -t(y-X_1))^2}{2(1-t) t}-\frac{(\ln [y/X_0] + 1/2)^2}{2 }}] \rd y}  { \int_0^\infty \exp[{-\frac{(D_t - tD_1 -t(y-X_1))^2}{2(1-t) t}-\frac{(\ln [y/X_0] + 1/2)^2}{2 }} ]/y \rd y}.
    \end{align}
    While there seems to be no explicit expression for $(M_t)_{t \in [0,1]}$, we can still simulate its sample-paths. For $t=1/2$, i.e., when the noise induced by the arcade process is maximised, we draw 10,000 points from $M_{1/2}$ and fit a kernel density estimator using Epanechnikov's kernel with cross-validated bandwidth, a lognormal$(\mu,s^2)$ density, and a lognormal$(-3/4,s^2)$ density to the data for comparison. Here are the results:
    \begin{figure}[H]
    \centering
    \includegraphics[width=.8\textwidth]{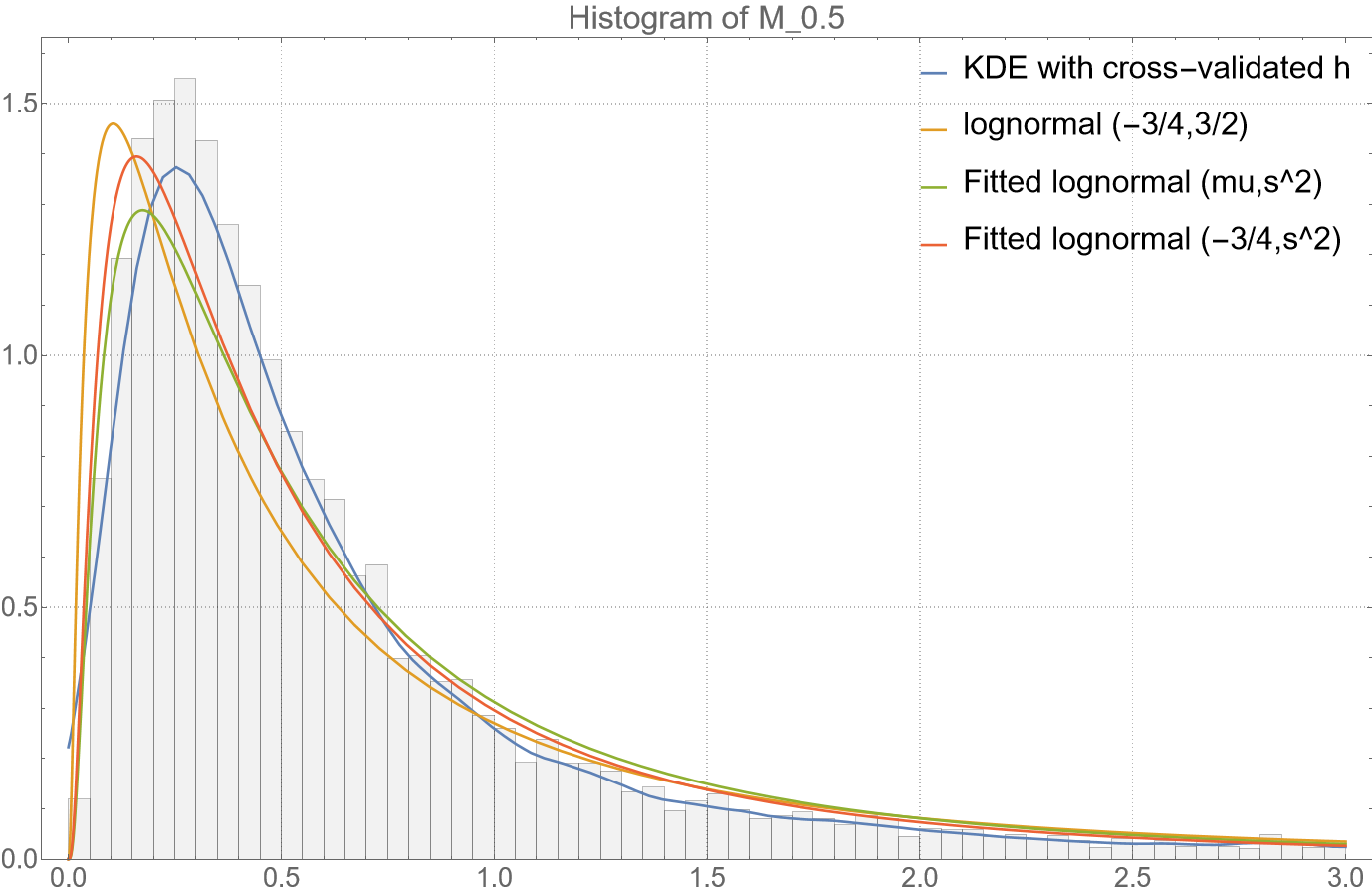}
    \caption{Histogram of 10,000 points drawn from $M_{1/2}$ with different densities fitted to these points for comparison. We focus on the interval $[0,3]$ for illustration purposes.}
    \end{figure}
    The empirical mean is about $1$ and the empirical variance is about $4.64$. The variance of $S_{3/2}$ is about $3.48$. While the data seem to have the ``overall shape'' of a lognormal density, we notice discrepancies once the lognormal densities are fitted. The reference density in yellow is the one of $S_{3/2} \sim$ lognormal$(-3/4,3/2)$. When we fit a lognormal$(\mu,s^2)$ density, we obtain $\mu \approx -0.67$ and $s^2 \approx 1.08$, and when we fit a lognormal$(-3/4,s^2)$ density, we obtain $s^2 \approx 1.08$. The $L^1$-norm between the reference lognormal density and the kernel density estimator is about $0.27$. The reference density does not seem to capture the data in a convincing way. Considering the number of data points, this suggests that $(M_t)_{t \in [0,1]}$ might not be equal in law to $(S_t)_{t \in [1,2]}$ since $S_{3/2} \sim$ lognormal$(-3/4,3/2)$. 

    Note that while Brownian motion for Gaussian targets and geometric Brownian motion for lognormal targets are the only known explicit solutions to MBB, one can efficiently approximate the solution for any irreducible measures using a Sinkhorn-like algorithm, see \cite{Joseph}. This would give us more room for experimentation to compare IB-MOT and MBB, once an efficient and general algorithm is developed for IB-MOT. A preliminary algorithm for empirical measures is introduced in the next section.
    \end{rem}
\begin{rem}
    Regardless of any connection between MBB and IB-MOT, it is clear that both are concerned with finding optimal Brownian martingales, i.e., martingales that can be written as integrals with respect to Brownian motion. This self-imposed restriction in MBB can theoretically be lifted as shown in \cite{Tschiderer}. The classical MBB seeks a martingale with given initial and terminal marginals whose transition kernel is as close to a Gaussian one as possible. Work \cite{Tschiderer} generalises this idea by replacing the Gaussian reference measure with an arbitrary probability measure $q$. The concept of $q$-Bass martingales (in our terminology, it would be $q$-ssBms) is introduced in discrete time. They are martingales whose transitions are shaped by the reference measure $q$ that is not necessarily Gaussian. Sufficient conditions are given under which such $q$-Bass martingales exist, depending on properties of the measure $q$ and the marginals. For IB-MOT, this kind of generalisation would not be possible without losing the framework introduced by FAMs. This would amount to changing in Proposition \ref{IB-MOTsol} the conditional measure $\mathcal N_x$ that appears in 
    \begin{equation}
        \inf\limits_{\gamma_x \in \mathcal M_x (\mu,\nu)} \int_\R \int_0^1( Q_{\gamma_x}(\alpha)-Q_{\mathcal N_x}(\alpha))^2 \rd \alpha \rd \mu(x),
    \end{equation}
    to a measure that is not Gaussian when integrated against the distribution of $X_0$, which would disconnect the problem from the innovations process $(W_t)$ and from the underlying FAM structure.
\end{rem}
\section{IB-MOT for empirical measures} \label{IBMOTempirical}
When considering empirical measures $\mu= 1/l \sum_{i=1}^l \delta_{x_i}$ and $\nu= 1/m \sum_{j=1}^m \delta_{y_j}$, where $(l,m)\in (\N_0 \setminus \{1\})^2$, $(x_1,\ldots,x_l)\in \R^l$, and $(y_1,\ldots,y_m)\in \R^m$, we can calculate the gradient of $\K_I$, allowing one to use gradient ascent methods to find $\pi^*_I$.
\begin{prop} \label{IBMOTprop2}
For the empirical measures $\mu$ and $\nu$, and denoting $p_{ij}:=\pi^X(\{x_i\},\{y_j\})=\P [ X_0= x_i, X_1=y_j]$ for $i=1,\ldots, l$, and $j=1,\ldots,m$, we have 
\begin{enumerate}
    \item 
    \begin{equation}
        \K_I(\pi^X)= \int_{T_0}^{T_1} \frac{S(\pi)\sqrt{H_1'(t)H_2(t)  - H_1(t) H_2'(t) }}{H_1(T_1)H_2(t)-H_1(t)H_2(T_1)} \rd t,
    \end{equation}
    where 
    \begin{equation}
        S(\pi)= \sum_{i=1}^l \sum_{j=1}^m \int_{-\infty}^\infty \left(y_j-M_t(g_0(t)x_i+g_1(t) y_j + a, x_i) \right)^2 f^{A_t^{(1)}}(a) p_{ij} \rd a,
    \end{equation}
    and $f^{A_t^{(1)}}$ is the density function of the Gaussian random variable $A_t^{(1)}$ for all $t \in (T_0,T_1)$.
    \item \begin{equation}
        \frac{\partial M_t}{\partial p_{uq}} (g_0(t) x_u + g_1(t) y_q + a, x_u) = \frac{f^{A_t^{(1)}}(a) (y_q - M_t(g_0(t) x_u + g_1(t) y_q + a, x_u))}{ \sum\limits_{j=1}^m f^{A_t^{(1)}}(a-g_1(t) (y_q-y_j)) p_{uj}},
    \end{equation}
    for $u=1,\ldots,l$ and $q=1,\ldots,m$.
    \item \begin{align}
        \frac{\partial S}{\partial p_{uh}}(\pi)&= \int_{-\infty}^\infty f^{A_t^{(1)}}(a) (y_h - M_t (g_0(t) x_u + g_1(t) y_h + a, x_u))^2 \rd a \nn \\
        &\quad -2 \sum_{q=1}^m \int_{-\infty}^\infty \frac{\left(f^{A_t^{(1)}}(a) (y_q - M_t(g_0(t) x_u + g_1(t) y_q + a, x_u))\right)^2 p_{uq}}{ \sum\limits_{j=1}^m f^{A_t^{(1)}}(a-g_1(t) (y_q-y_j)) p_{uj}} \rd a,
    \end{align}
    for $u=1,\ldots,l$ and $h=1,\ldots,m$.
\end{enumerate}
\end{prop}
\begin{proof}
\begin{enumerate}
    \item The expression follows from 
    \begin{align}
        &\E \left[(X_1-M_t(X_0,I_t^{(1)}))^2  \right] \nn \\
        &\hspace{1.5cm}=  \int_{\R^3} (y-M_t(x, g_0(t) x + g_1(t) y + a )^2 \rd F^{(X_0,X_1,A_t^{(1)})}(x,y,a) \nn \\
        &\hspace{1.5cm}= \int_{\R^3} (y-M_t(x, g_0(t) x + g_1(t) y + a )^2 f^{A_t^{(1)}}(a) \rd \pi (x,y) \nn \\
        &\hspace{1.5cm}= \sum_{i=1}^l \sum_{j=1}^m \int_{-\infty}^\infty \left(y_j-M_t(g_0(t)x_i+g_1(t) y_j + a, x_i) \right)^2 f^{A_t^{(1)}}(a) p_{i,j} \rd a,
    \end{align}
    where we use the fact that $(A_t^{(1)})$ is independent of $(X_0,X_1)$.
    \item Recall that
    \begin{equation}
        M_t (x_u, g_0(t) x_u + g_1(t) y_q + a) = \frac{\sum\limits_{j=1}^m y_j f^{I_t^{(1)} | X_0=x_u,X_1=y_j} (g_0(t) x_u + g_1(t) y_q + a) p_{uj}  }{\sum\limits_{j=1}^m  f^{I_t^{(1)} | X_0=x_u,X_1=y_j} (g_0(t) x_u + g_1(t) y_q + a) p_{uj}},
    \end{equation}
    where $f^{I_t^{(1)} | X_0=x_u,X_1=y_j} (g_0(t) x_u + g_1(t) y_q + a) = f^{A_t^{(1)}} (a - g_1(t)(y_q-y_j))$. The derivative of the numerator with respect to $p_{uq}$ is $y_q f^{A_t^{(1)}} (a)$ and the derivative of the denominator with respect to $p_{uq}$ is $f^{A_t^{(1)}} (a)$. Hence,
    \begin{align}
        &\frac{\partial M_t}{\partial p_{uq}} (g_0(t) x_u + g_1(t) y_q + a, x_u) \nn \\
        &= \frac{y_q f^{A_t^{(1)}} (a) \sum\limits_{j=1}^m f^{A_t^{(1)}} (a - g_1(t)(y_q-y_j)) p_{uj}   }{\left(\sum\limits_{j=1}^m f^{A_t^{(1)}} (a - g_1(t)(y_q-y_j)) p_{uj} \right)^2 }  \nn \\
        &\hspace{5cm}- \frac{ f^{A_t^{(1)}} (a) \sum\limits_{j=1}^m y_j f^{A_t^{(1)}} (a - g_1(t)(y_q-y_j)) p_{uj}  }{\left(\sum\limits_{j=1}^m f^{A_t^{(1)}} (a - g_1(t)(y_q-y_j)) p_{uj} \right)^2} \nn\\
        &= \frac{y_q f^{A_t^{(1)}} (a)  }{\sum\limits_{j=1}^m f^{A_t^{(1)}} (a - g_1(t)(y_q-y_j)) p_{uj} } - \frac{ f^{A_t^{(1)}} (a)  M_t (x_u, g_0(t) x_u + g_1(t) y_q + a) }{\sum\limits_{j=1}^m f^{A_t^{(1)}} (a - g_1(t)(y_q-y_j)) p_{uj} }\nn \\
        &= \frac{f^{A_t^{(1)}}(a) (y_q - M_t(g_0(t) x_u + g_1(t) y_q + a, x_u))}{ \sum\limits_{j=1}^m f^{A_t^{(1)}}(a-g_1(t) (y_q-y_j)) p_{uj}}.
    \end{align}
    \item Differentiating, we obtain 
    \begin{dmath}
        \frac{\partial S}{\partial p_{uh}}(\pi)
         = \int_{-\infty}^\infty f^{A_t^{(1)}}(a){(y_h - M_t(g_0(t) x_u + g_1(t) y_h + a, x_u))} \\
          \left({(y_h - M_t(g_0(t) x_u + g_1(t) y_h + a, x_u))} - 2 p_{uh} \frac{\partial M_t}{\partial p_{uh}} (g_0(t) x_u + g_1(t) y_h + a, x_u)\right) \rd a 
         - 2 \sum_{q=1,q\neq h}^m \int_{-\infty}^\infty f^{A_t^{(1)}}(a)p_{uq}{(y_q - M_t(g_0(t) x_u + g_1(t) y_q + a, x_u))} 
         \frac{\partial M_t}{\partial p_{uq}} (g_0(t) x_u + g_1(t) y_q + a, x_u) \rd a.
    \end{dmath}
    We inject the expression of $\frac{\partial M_t}{\partial p_{uq}} (g_0(t) x_u + g_1(t) y_q + a, x_u)$, obtained in the second part of the proof, in the above equation and rewrite it as follows
    \begin{align}
        &\frac{\partial S}{\partial p_{uh}}(\pi)\nn\\
        & = \int_{-\infty}^\infty f^{A_t^{(1)}}(a)(y_h - M_t(g_0(t) x_u + g_1(t) y_h + a, x_u))^2\nn \\ 
        & \hspace{5.5cm} \left( 1 - \frac{2 p_{uh} f^{A_t^{(1)}}(a)}{\sum\limits_{j=1}^m f^{A_t^{(1)}}(a-g_1(t) (y_q-y_j)) p_{uj}}\right) \rd a \nn \\
        &\quad -2 \sum_{q=1,q\neq h}^m \int_{-\infty}^\infty \frac{(f^{A_t^{(1)}}(a))^2 p_{uq}(y_q - M_t(g_0(t) x_u + g_1(t) y_q + a, x_u))^2}{\sum\limits_{j=1}^m f^{A_t^{(1)}}(a-g_1(t) (y_q-y_j)) p_{uj}} \rd a.
    \end{align}
    We notice that the missing term in the sum, i.e., the term when $q=h$, appears in the first integral, which allows us to write
    \begin{align}
        &\frac{\partial S}{\partial p_{uh}}(\pi)\nn\\
        &= \int_{-\infty}^\infty f^{A_t^{(1)}}(a)(y_h - M_t(g_0(t) x_u + g_1(t) y_h + a, x_u))^2 \rd a \nn \\
        &\quad - 2 \sum_{q=1}^m \int_{-\infty}^\infty \frac{(f^{A_t^{(1)}}(a))^2 p_{uq}(y_q - M_t(g_0(t) x_u + g_1(t) y_q + a, x_u))^2}{\sum\limits_{j=1}^m f^{A_t^{(1)}}(a-g_1(t) (y_q-y_j)) p_{uj}} \rd a.
    \end{align}
\end{enumerate}
\end{proof}
As implied by the expressions of the partial derivatives, calculating analytically the argument maximising $\K_I$ is unlikely. We turn to numerical methods to approximate the solution. To get an $\epsilon$-approximation of the solution $\pi^*_I$, one can apply a constrained modification of the gradient ascent to $\K_I$. For background on gradient methods, see Chapter 14 in \cite{Guler}. Let $\epsilon>0$, 
\begin{equation}
    \delta=\max_{\pi_1, \pi_2 \in \mathcal M \left(\frac{1}{l} \sum\limits_{i=1}^l \delta_{x_i}, \frac{1}{m} \sum\limits_{j=1}^m \delta_{y_j} \right)} \sqrt{\sum_{i=1}^{l}\sum_{j=1}^{m} (\pi_1(\{x_i\},\{y_j\}) -  \pi_2(\{x_i\},\{y_j\}))^2},
\end{equation}
$\Theta=4\left \lceil{\delta^2 \norm{\nabla \K_I}_{\infty}^2 / \epsilon^2}\right \rceil $, and $\lambda= \epsilon/2\norm{\nabla \K_I}_{\infty}^2$. These variables determine the diameter of the set $\mathcal M \left(1/l\sum_{i=1}^l \delta_{x_i}, 1/m\sum_{j=1}^m \delta_{y_j} \right)$, the number of steps of the gradient ascent, and the learning rate, respectively. We start at any point $\pi^{(0)} \in \mathcal M \left(1/l\sum_{i=1}^l \delta_{x_i}, 1/m\sum_{j=1}^m \delta_{y_j} \right)$, and for $\theta=1,\ldots, \Theta$, we set 
\begin{align}
& \hat \pi ^{(\theta)} = \pi ^{(\theta-1)}+\lambda \nabla \K_I \left(\pi^{(\theta-1)}\right), \\
& \pi^{(\theta)} = \operatorname{Proj}_{\mathcal M \left(\frac{1}{l}\sum\limits_{i=1}^l \delta_{x_i}, \frac{1}{m}\sum\limits_{j=1}^m \delta_{y_j} \right)}\left(\hat \pi ^{(\theta)}\right),
\end{align}
where $\pi^{(\theta)}$ is the Euclidean projection of $\hat \pi ^{(\theta)}$ onto the set of martingale couplings. This means that every time the gradient leads to a point outside of the set $\mathcal M \left(1/l\sum_{i=1}^l \delta_{x_i}, 1/m\sum_{j=1}^m \delta_{y_j} \right)$, $\pi^{(\theta)}$ projects that point back onto that set. The $\epsilon$-approximation of the solution $\pi^*_I$ is then $1/\Theta \sum_{\theta=1}^\Theta \pi^{(\theta-1)}$ as we shall see below. In the next proposition, we use the following short-hand notation: If a coupling $\pi$ appears in a norm or a scalar product, this means that we consider the vector of $l\times m$ elements $\pi(\{x_i\},\{y_j\})$ for $i=1,\ldots, l$ and $j=1,\ldots,m$. 
\begin{prop} \label{IBMOTprop3}
    It holds that $ \K_I(\pi_I^*)-\K_I\left(\frac{1}{\Theta} \sum_{\theta=1}^\Theta \pi^{(\theta-1)} \right) \leqslant \epsilon$.
\end{prop}
\begin{proof}
First, we notice that for $\theta=1,\ldots,\Theta$,
\begin{align}
    \norm{\pi_I^*-\pi^{(\theta)}}^2 &\leqslant  \norm{\pi_I^*-\hat \pi^{(\theta)}}^2 \nn \\
    &= \norm{\pi_I^*-\pi^{(\theta-1)}-\lambda \nabla \K_I\left(\pi^{(\theta-1)}\right)}^2 \nn \\
    &= \norm{\pi_I^*-\pi^{(\theta-1)}}^2 + \lambda^2\norm{ \nabla \K_I\left(\pi^{(\theta-1)}\right)}^2 \nn \\
    &\hspace{5cm}-2 \lambda\left\langle\nabla \K_I\left(\pi^{(\theta-1)}\right), \pi_I^*-\pi^{(\theta-1)}\right\rangle \nn \\
    & \leqslant \norm{\pi_I^*-\pi^{(\theta-1)}}^2 + \lambda^2\norm{\nabla \K_I}_{\infty}^2 -2 \lambda\left\langle\nabla \K_I\left(\pi^{(\theta-1)}\right), \pi_I^*-\pi^{(\theta-1)}\right\rangle.
\end{align}
Hence,
\begin{align}
    \sum_{\theta=1}^\Theta \left\langle\nabla \K_I\left(\pi^{(\theta-1)}\right), \pi_I^*-\pi^{(\theta-1)}\right\rangle &\leqslant \frac{\lambda \Theta}{2} \norm{\nabla \K_I}_{\infty}^2 + \frac{1}{2\lambda} \sum_{\theta=1}^\Theta \norm{\pi_I^*-\pi^{(\theta-1)}}^2 \nn \\
    & \hspace{6cm} -\norm{\pi_I^*-\pi^{(\theta)}}^2\nn \\
    &= \frac{\Theta \epsilon}{4} +  \frac{1}{2\lambda} \norm{\pi_I^*-\pi^{(0)}}^2- \frac{1}{2\lambda}\norm{\pi_I^*-\pi^{(\Theta)}}^2 \nn \\
    & \leqslant \frac{\Theta \epsilon}{4} +  \frac{1}{2\lambda} \norm{\pi_I^*-\pi^{(0)}}^2 \nn \\
    & \leqslant \frac{\Theta \epsilon}{4} +  \frac{\delta^2}{2\lambda}.
\end{align}
Since $\K_I(\pi_I^*)-\K_I(\pi^{(\theta-1)})\leqslant \left\langle\nabla \K_I\left(\pi^{(\theta-1)}\right), \pi_I^*-\pi^{(\theta-1)}\right\rangle$, we have that
\begin{align}
    \K_I(\pi_I^*)-\K_I\left(\frac{1}{\Theta} \sum_{\theta=1}^\Theta \pi^{(\theta-1)} \right) &\leqslant \K_I(\pi_I^*)-\frac{1}{\Theta} \sum_{\theta=1}^\Theta \K_I(\pi^{(\theta-1)}) \nn \\
    &= \frac{1}{\Theta}\left(\sum_{\theta=1}^\Theta \K_I(\pi_I^*)- \K_I(\pi^{(\theta-1)}) \right) \nn \\
    &\leqslant \frac{ \epsilon}{4} +  \frac{\delta^2}{2\lambda \Theta} \nn \\
    &\leqslant\epsilon.
\end{align}
\end{proof}
\begin{rem}
    The choice of the $\epsilon$-approximation $1/\Theta \sum_{\theta=1}^\Theta \pi^{(\theta-1)}$ is arbitrary. There are other candidates that satisfy the same property. For instance, we could have chosen the maximising projection, i.e., $\argmax_{\theta=1,\ldots, \Theta} \K_I(\pi^{(\theta-1)})$, instead. However, the average projection presents a practical advantage, since applying the algorithm depends upon the computation of many integrals at each step. These integrals require numerical methods that will result in approximation errors. Hence, by averaging the projections, the approximation errors are averaged, which avoids that a single projection may perform very well due to the error it contains. 
\end{rem}
\begin{rem}
    A good choice for $\pi^{(0)}$ is the projection of the independent coupling onto $\mathcal M \left(1/l\sum_{i=1}^l \delta_{x_i}, 1/m\sum_{j=1}^m \delta_{y_j} \right)$. This is because, regardless of the empirical measures considerations that we are treating here, if we were to remove the martingale condition in IB-MOT, the solution would always be the independent coupling, i.e., $\argmax_{\pi \in \Pi(\mu,\nu)} \K_I(\pi)$ is the independent coupling for $\mu$ and $\nu$.
\end{rem}
\begin{rem}
    Calculating an $\epsilon$-approximation of the solution $\pi^*_I$ is usually a hard problem to solve. At each iteration, we have to evaluate the gradient, which involves numerical integration, then we must compute the projection, which is a quadratic problem with linear constraints. An efficient way of computing this projection is to use the interior point methods (IPMs); see \cite{Wright}. These methods transform the constrained problem into a series of unconstrained problems using barrier functions. They iteratively move through the interior of the feasible region toward the optimal point. The total time complexity of a primal-dual path-following IPM is approximately
    \begin{equation*}
        O\left((lm)^{3.5} \log \left(\frac{1}{\tilde \epsilon}\right)\right),
    \end{equation*}
    where $\tilde \epsilon$ is the desired accuracy of the projection. This has to be repeated at worst $\Theta$ times.
\end{rem}
\begin{rem} \label{convexrem}
    What happens if $\mu$ and $\nu$ are not empirical measures? For instance, continuous measures? Could one sample enough points from $\mu$ and $\nu$, apply the algorithm, and obtain an empirical version of the optimal coupling for $\mu$ and $\nu$? Unfortunately, it is not that straightforward. The fact that $\mu$ and $\nu$ are in convex order does not imply that their empirical counterparts $\widehat \mu =1/l \sum_{i=1}^l \delta_{x_i}$ and $\widehat \nu=1/m\sum_{j=1}^m \delta_{y_i}$ will be in convex order. In fact, the probability of them being in convex order is $0$. This is why one needs to "convexify" the order of $\widehat \mu$ and $\widehat \nu$ before applying IB-MOT, while making sure that the convexification procedure does not change the empirical measures too much. One way of achieving this goal is to look for measures $\widetilde \mu= 1/ \widetilde l \sum_{i=1}^{\widetilde l} \delta_{\widetilde x_i}$ and $\widetilde \nu= 1/ \widetilde m \sum_{j=1}^{\widetilde m} \delta_{\widetilde y_j}$ that are in convex order, and close to $\widehat \mu$ and $\widehat \nu$ in Wasserstein metric, respectively:
    \begin{equation}
        \inf_{(\widetilde \mu, \widetilde \nu)\,  \in\,  \mathcal C} \frac{1}{2} \left(\alpha W_1(\widehat \mu, \widetilde \mu) + \beta W_1(\widehat \nu, \widetilde \nu) \right),
    \end{equation}
    where $\alpha$ and $\beta$ are positive real numbers such that $1/\alpha + 1/\beta =1$, and $\mathcal C = \{(\mu,\nu)\in \mathcal P^1(\R)\times \mathcal P^1(\R) \mid  \exists (l,m) \in \N^2_0, (x_1,\ldots, x_l) \in \R^l, (y_1,\ldots, y_m) \in \R^m \text{ such that }  \mu = 1/l\sum_{i=1}^l \delta_{x_i}, \nu=1/m\sum_{j=1}^m \delta_{y_i}, \mu\leqslant_{\mathrm{cx}} \nu\}$. The constants $\alpha$ and $\beta$ inform us of the importance of keeping $\widetilde \mu$ close to $\widehat \mu$ and $\widetilde \nu$ close to $\widehat \nu$, respectively. Let $Q^{\widehat \mu}$, $Q^{\widetilde \mu}$, $Q^{\widehat \nu}$, $Q^{\widetilde \nu}$ be the quantile functions of $\widehat \mu, \widetilde \mu, \widehat \nu$, and $\widetilde \nu$, respectively. We recall that 
    \begin{equation}
        \alpha W_1(\widehat \mu, \widetilde \mu) + \beta W_1(\widehat \nu, \widetilde \nu) = \alpha \int_0^1 \abs{Q^{\widehat \mu}(x) - Q^{\widetilde \mu}(x)} \rd x  + \beta \int_0^1 \abs{Q^{\widehat \nu}(x) - Q^{\widetilde \nu}(x)} \rd x,
    \end{equation}
    and that $\widehat \mu\leqslant_{\mathrm{cx}} \widehat \nu$ if and only if $Q(t):= \int_0^t Q^{\widehat \mu}(x) - Q^{\widehat \nu}(x) \rd x \geqslant 0$ and $Q(1)=0$. Since we do not have $\widehat \mu\leqslant_{\mathrm{cx}} \widehat \nu$, we are looking for a continuous, piecewise affine function $F:[0,1]\rightarrow \R$ such that 
    \begin{enumerate}
        \item $Q(t)   \geqslant F(t)$, $F(0)=0$, and $F(1)= Q(1)=\E[\widehat \mu]-\E[\widehat \nu]$,
        \item there exists a real function $f$ such that $F(t)=\int_0^t f(x)  \rd x$. This means that $Q(t) - F(t) = \int_0^t Q^{\widehat \mu}(x) -  f(x)/\alpha- (Q^{\widehat \nu}(x) +  f(x)/\beta ) \rd x \geqslant 0$,
        \item the functions $Q^{\widehat \mu}(x) -   f(x)/\alpha$ and $Q^{\widehat \nu}(x) +  f(x)/\beta$ are quantile functions, i.e., càglàd non-decreasing functions on $[0,1]$. These functions are the candidates for $Q^{\widetilde \mu}$ and $Q^{\widetilde \nu}$, respectively. Equivalently, $f$ is càglàd on $[0,1]$, and the functions $t \rightarrow \int_0^t Q^{\widehat \mu}(x)\rd x -  F(t)/\alpha $ and $t \rightarrow \int_0^t Q^{\widehat \nu}(x)\rd x +  F(t)/\beta $ are convex.
    \end{enumerate}
    The set of functions that satisfy all these constraints is not empty since $Q$ satisfies them all. Hence, the original problem $\inf_{(\widetilde \mu, \widetilde \nu)\,  \in\,  \mathcal C} 1/2 \left(\alpha W_1(\widehat \mu, \widetilde \mu) + \beta W_1(\widehat \nu, \widetilde \nu) \right)$ becomes 
    \begin{equation}
        \inf_{f} \frac{1}{2} \left(\alpha \int_0^1 \abs{f(x)/\alpha } \rd x  + \beta \int_0^1 \abs{ f(x)/\beta} \rd x \right) = \inf_{f} \int_0^1 \abs{f(x)}\rd x= \inf_{f} \norm{f}_{L^1}
    \end{equation}
    under the constraints on $f$ listed above. We can choose $f$ to be the left-derivative of $F$. This will ensure that $f$ is càglàd on $[0,1]$. Let $\Delta f (x) := \lim_{y \rightarrow x^+} f(y) - f(x)$. Then, the third condition holds if and only if $\Delta Q^{\widehat \mu}(x) \geqslant \Delta f (x) / \alpha$ and $\Delta Q^{\widehat \nu}(x) \geqslant - \Delta f (x) / \beta$ for all $x\in (0,1]$. To minimise $\norm{f}_{L^1}$ under the listed constraints, we notice that the following properties must hold:
    \begin{enumerate}
        \item $\min F = \min Q$ and $\{t\in[0,1] \mid F(t)= \min F\}=\{t\in[0,1] \mid Q(t)= \min Q\}$. Otherwise, there exists a function between $F$ and $Q$ that yields a smaller $\norm{f}_{L^1}$. Hence, when combining this fact with the first condition, $F$ must pass through the points $(0,Q(0))$, $\{(x,Q(x)) \mid x \in \{ \argmin Q\} \}$, and $(1,Q(1))$ while being under $Q$.
        \item $F$ must be convex on $[0,1]$. Otherwise, the convex envelope of $F$ (the largest convex function under $F$) yields a smaller $\norm{f}_{L^1}$. This implies that $\Delta f(x) \geqslant 0$ for all $x\in (0,1]$, and $\Delta f (x) = 0$ whenever $\Delta Q^{\widehat \mu}(x)=0$, and that the second part of the third condition is automatically satisfied.
    \end{enumerate}
    Combining all these facts, we see that $F$ should be equal to the convex envelope of $Q$. If $F$ was a convex function passing through the points $(0,Q(0))$, $\{(x,Q(x)) \mid x \in \{ \argmin Q\} \}$, $(1,Q(1))$ while being under the convex envelope of $Q$, then $\Delta Q^{\widehat \mu}(x) \geqslant \Delta f (x) / \alpha$ for all $x\in (0,1]$ cannot hold. Whereas if it was a function passing through the points $(0,Q(0))$, $\{(x,Q(x)) \mid x \in \{ \argmin Q\} \}$, $(1,Q(1))$ while strictly being between $Q$ and its convex envelope, then $F$ cannot be convex. Hence, $\widehat \mu$ and $\widehat \nu$ are the empirical measures defined by the quantile functions $Q^{\widehat \mu}(x) -   f(x)/\alpha$ and $Q^{\widehat \nu}(x) +  f(x)/\beta$ respectively, where $f$ is the left-derivative of the convex envelope of $Q$.
    
    In the case $\alpha=1$, $\beta=0$, this problem of "convexifying" the order of one-dimensional empirical measures by minimising the first Wasserstein metric is closely related to the problem studied in \cite{Alphonsi}. In this paper, the authors consider the problem
    \begin{equation} \label{alphonsip}
        \inf_{\widetilde \mu} W_p(\mu, \widetilde \mu),
    \end{equation}
    where $p\in \N_0$, $\mu$ and $\nu$ are continuous measures on $\R^n$ instead of $\R$ and are not in convex order, and the optimisation takes place in the set of measures which are in convex order with $\nu$. The only conceptual difference between this problem and the problem we described is that we have continuous measures $\mu$ and $\nu$ in convex order and we are "fixing" a sampled $\widehat \mu$ from $\mu$ that is not in convex order with a sampled $\widehat \nu$ from $\nu$, whereas in \cite{Alphonsi}, $\mu$ and $\nu$ are not in convex order and they want to find the closest measure to $\mu$ that is in convex order with $\nu$. Setting $\alpha=1$, $\beta=0$ in our solution and $n=1, p=1$ in the solution described in \cite{Alphonsi} makes both solutions coincide up to this conceptual difference, meaning that their solution is also based on the convex envelope of the integrated difference of the quantile functions (of $\mu$ and $\nu$ instead of $\widehat \mu$ and $\widehat \nu$). A key insight is that the solution is actually independent from the choice of $p\in \N_0$ when $n=1$. The heuristic described here that leads to our solution, which relies on geometric interpretation, only works in 1 dimension, the setting of interest in this thesis. The problem $(\ref{alphonsip})$ is solved for $p=2$ and $n>1$ in \cite{Alphonsi} and the solution can be described as follows. Let $\phi: \mathbb{R}^d \rightarrow \mathbb{R}$ be a convex function such that $\nabla \phi \#  \mu= \nu$, i.e., $\phi$ is the Brenier potential transporting $ \mu$ to $ \nu$, and define $\psi$ as the convex envelope of $\phi$. The solution is then the measure $\nabla \psi \#  \mu$. The choice of the second Wasserstein metric is not arbitrary, it is deeply tied to the geometry of optimal transport and the structure of convex order projections:
    \begin{itemize}
        \item In $\R^n$, the optimal transport map between absolutely continuous measures under the $W_2$ metric is the gradient of a convex function. This is not true for $W_p$ with $p \neq 2$, where optimal maps may not be gradients of convex functions.
        \item The projection onto the convex order set in higher dimensions relies on convexification of the Brenier potential, which is specific to the quadratic cost. For $p \neq 2$, there is no analogous potential theory that yields such a clean projection structure.
        \item The space $\mathcal{P}_2\left(\mathbb{R}^d\right)$ endowed with $W_2$ has a Riemannian-like geometry that allows for orthogonal projections onto convex sets. This geometric structure is essential for computing the projection.
    \end{itemize}
    For $p \neq 2$, the projection problem is still well-posed (i.e., one can define the closest measure in convex order for $W_p$), but:
    \begin{itemize}
        \item The projection is not characterised by a convex envelope of a potential.
        \item There is no known explicit or semi-explicit solution.
        \item The problem becomes much harder both analytically and computationally.
    \end{itemize}
\end{rem}
\section{Conclusion}
Filtered arcade martingales (FAMs) interpolate between given random variables, which form a discrete-time martingale, by estimating the value of the final random variable in the chosen sequence based on the information that is unveiled over time by a randomised arcade process. Inspired by how noise is injected in optimal transport, we use FAMs to inject noise in the martingale counterpart to optimal transport. This results in the information-based martingale optimal transport (IB-MOT) problem, which seeks to find the supremum of the expectation of the integrated, weighted and squared difference between the final target random variable and the considered FAM over the set of martingale couplings. 

While at first glance, the problem appears substantially different from familiar optimal transport problems, we show that the IB-MOT problem can be rewritten using the innovations process of the underlying FAM. This new formulation transforms the objective function of the IB-MOT problem into a more recognisable form, while at the same time revealing the stochastic filtering aspect underpinning the FAM in terms of which the IB-MOT problem is posed.

We prove that the supremum in the IB-MOT objective function is indeed a maximum, and that the martingale coupling that achieves this maximum is unique. This result is used to show that the optimal FAM, which is the FAM constructed using the solution of the associated IB-MOT problem, between the marginals of Brownian motion, is Brownian motion, when the selected RAP is a randomised anticipative Brownian bridge. Furthermore, when considering real empirical measures, the IB-MOT objective function admits simplifications which enable one to calculate explicitly its gradient, thus opening the door to gradient ascent methods for estimating the optimal solution. We propose an algorithm that averages the Euclidean projections (on the set of martingale couplings of the considered empirical measures) of the different couplings obtained by following the direction of the gradient, starting with the independent coupling.

Sampling points from continuous distributions and applying the algorithm for empirical measures to obtain an empirical version of the optimal coupling is a challenge; the probability of the sampled measures being in convex order is zero, even if the original continuous distributions are in convex order. Hence, we discuss a method, itself based on optimal transport, to convexify the order of the sampled measures while modifying the samples minimally.

\section*{Acknowledgments}
A. Macrina is grateful for the support by the Fields Institute for Research in Mathematical Sciences through the awards of a Fields Research Fellowship in 2022 and an Elliott-Yui Distinguished Visitorship in 2025. The contents of this manuscript are solely the responsibility of the authors and do not necessarily represent the official views of the Fields Institute. G. Kassis acknowledges the UCL Department of Mathematics for a Teaching Assistantship Award. Furthermore, the authors thank B. Acciaio and G. Pammer for pointers to optimal transport and stretched Brownian motion, S. Cohen, J. Guyon, F. Krach, J. Obłój, G. W. Peters, and T.-K. L. Wong for useful conversations and suggestions. Attendees of the 7th International Conference Mathematics in Finance---MiF Kruger Park---(July 2023, South Africa), the Talks in Financial and Insurance Mathematics at ETH Zurich (Nov. 2023, Switzerland), the Department of Mathematics Seminar of Ritsumeikan University (Nov. 2023, Japan), and participants in the FAMiLLY Workshop at the University of Liverpool (Dec. 2023, U.K.), the XXV Workshop on Quantitative Finance (Apr. 2024, Italy), the Research Seminar of the Department of Mathematics \& Statistics, University of Ottawa (Nov. 2024, Canada), the Quantitative Methods in Finance (QMF) 2024 International Conference at the University of Technology Sydney (Dec. 2024, Australia), and in the 69th Annual Meeting of the Australian Mathematical Society (Dec. 2025, Australia) are thanked for their comments and suggestions.  

\end{document}